\documentclass[10pt, twoside]{amsart}

\usepackage[pagewise]{lineno}

\usepackage{xpatch}
\makeatletter   
\xpatchcmd{\@tocline}
{\hfil\hbox to\@pnumwidth{\@tocpagenum{#7}}\par}
{\ifnum#1<0\hfill\else\dotfill\fi\hbox to\@pnumwidth{\@tocpagenum{#7}}\par}
{}{}
\makeatother    

\makeatletter
 \def\l@subsection{\@tocline{2}{0pt}{4pc}{6pc}{}}
\def\l@subsubsection{\@tocline{3}{0pt}{8pc}{8pc}{}}
 \makeatother

\usepackage[utf8]{inputenc}
\usepackage{hyperref}
\newcommand{\subscript}[2]{$#1 _ #2$}

\usepackage{graphicx}              
\usepackage{amsmath, bm}               
\usepackage{amsfonts}              
\usepackage{amsthm}                
\usepackage{amssymb}
\usepackage{mathtools}             

\usepackage{enumitem}

\usepackage[all]{xy}
\usepackage{amscd}

\DeclareSymbolFont{symbolsC}{U}{pxsyc}{m}{n}
\DeclareMathSymbol{\coloneqq}{\mathrel}{symbolsC}{"42}
\allowdisplaybreaks
\usepackage{color}
\usepackage{fp}
\newcommand{\N}{\Z_{\geq 0}}

\newcommand{\Z}{\mathbb{Z}}

\newcommand{\C}{\mathbb{C}}
\newcommand{\R}{\mathbb{R}}

\newcommand{\Hi}{\mathcal{H}}
\newcommand{\Li}{\mathcal{B}}

\newcommand{\cj}{\overline}
\newcommand{\op}{\mathrm}
\newcommand{\dd}{\mathrm{dist}}

\newcommand{\LL}{{\mathbb{L}}}

\DeclarePairedDelimiterX{\normb}[1]{\lVert}{\rVert}{#1}

\numberwithin{equation}{section}
\newtheorem{theorem}{Theorem}[section]
\newtheorem{lemma}[theorem]{Lemma}
\newtheorem{proposition}[theorem]{Proposition}
\newtheorem{corollary}[theorem]{Corollary}

\let\origproofname\proofname
\renewcommand{\proofname}{\upshape\textbf{\origproofname}}

\newcommand{\vf}{\mathbf}

\theoremstyle{definition}

\newtheorem{definition}[theorem]{Definition}
\newtheorem{remark}[theorem]{Remark}

\newtheorem*{theorem*}{Theorem}

\makeatletter
\newtheorem*{rep@theorem}{\rep@title}
\newcommand{\newreptheorem}[2]{%
\newenvironment{rep#1}[1]{%
 \def\rep@title{#2 \ref{##1}}%
 \begin{rep@theorem}}%
 {\end{rep@theorem}}}
\makeatother

\newreptheorem{theorem}{Theorem}

\usepackage{indentfirst}

\begin{document}

\title{$L^p$-spectral Triples and $p$-Quantum Compact Metric Spaces }

\author{Alonso Delfín}
\email{alonso.delfin@colorado.edu}
\urladdr{https://math.colorado.edu/\symbol{126}alde9049}

\author{Carla Farsi}
\email{carla.farsi@colorado.edu}
\urladdr{https://www.colorado.edu/math/carla-farsi}

\author{Judith Packer}
\email{judith.jesudason@colorado.edu}
\urladdr{https://math.colorado.edu/\symbol{126}packer}
\address{Department of Mathematics, University of Colorado, Boulder, CO 80309-0395}

\date{\today}

\subjclass[2020]{Primary 46H15, 46H35, 46L89, 58B34; Secondary 43A15}
\thanks{\textsc{Department of Mathematics, University of Colorado, Boulder CO 80309-0395, USA}
}
\keywords{Spectral triples, Dirac Operator, Length Functions, UHF algebras, $L^p$-operator algebras, Quantum metrics}

\maketitle

\begin{abstract}
For $p \in [1, \infty)$, we generalize the concept of classical spectral triples
by extending the framework from Hilbert spaces to $L^p$-spaces, and from C*-algebras to $L^p$-operator algebras. In addition, we define an $L^p$-spectral triple to be metric when the state space of the algebra has a $p$-quantum compact metric space structure. 
Specifically, we construct $L^p$-spectral triples for reduced $L^p$-group algebras of countable discrete groups with proper length functions and also for $L^p$ UHF-algebras of infinite tensor product type, the latter inspired by E.\ Christensen and C.\ Ivan's construction of a Dirac operator on AF C*-algebras. We prove that $L^p$-spectral triples associated with $L^p$-group algebras (provided that the length function is of bounded doubling) and those associated with $L^p$ UHF-algebras are always metric.
\end{abstract}

\tableofcontents

\section{Introduction}

In the early 20th century, the study of linear operators and function spaces became important in both mathematics and physics.  The work of S.\ Banach, R.\ M.\ Fr\'ech\'et and D.\ Hilbert culminated in S.\ Banach's formulation of abstract complete normed spaces. A decade later, in the 1930's and 1940's, operator algebras were developed by J.\ von Neumann and then I.\ Gelfand, M.\ Naimark, and I.\ Segal, in part because of the need to provide a rigorous mathematical foundation for quantum mechanics. In operator algebra theory, the study of general Banach algebras held a prominent role,  even though  sub-branches of this theory (pertinent to C*- and von Neumann algebras) quickly gained importance.  Quantized states appeared as objects belonging to the  state spaces of these structures, which are, roughly speaking, spaces of  linear functionals. In addition, analysis on locally compact groups also played a prominent role, since groups are often used to  represent symmetries of  spaces of operators. 

For $p\in [1, \infty)$, an $L^p$-operator algebra is a Banach algebra 
which admits an isometric representation on an  $L^p$-space. Representations of groups on $L^p$-spaces and the algebras they generate is a highly active area of current research and is regarded as one of the central branches of harmonic analysis. Research in the area was largely started by C. S. Herz in \cite{CSH73, CSH71}, who studied $L^p$-operator algebras generated by the left regular representation of a locally compact group. 
In the 2010's, N.\ C.\ Phillips re-initiated their study, leading a program aimed at generalizing the modern theory of C*-algebras to  $L^p$-operator algebras, or outlining their differences.  Since their reintroduction, $L^p$-operator algebras have since seen a growing amount of interest among researchers, see for example 
\cite{NCP13_2, ncp2012AC, NCP13, Choi15, Gardella15, GarThi15, GarLup16,ncpmgv2017AF, Chung18, Cortinas19, Gardella21, Choi24}. Part of this trend has been marked by the extension to their $L^p$-analogues of various constructions traditionally associated with $C^*$-algebras. Notably, algebras such as Cuntz algebras \cite{ncp2012AC}, group and groupoid algebras \cite{NCP13, GL17}, crossed products \cite{NCP13}, and those arising from finite-dimensional constructions such as AF and UHF-algebras \cite{NCP13_2, ncpmgv2017AF}, have all found generalizations in the context of $L^p$-spaces.

Spectral triples, whose origin stems from an axiomatic characterization of the notion of  unbounded selfadjoint operators with compact resolvent on Hilbert spaces, represent a significant construction in the realm of noncommutative
geometry.  Their study was initiated by A.\ Connes in the late 1980's \cite{Connes89}. These triples provide a framework for extracting geometric data from operator algebras, particularly through the quantum metric space structure they induce on the state space of their C*-algebras via constructions pioneered by A.\ Connes and M.\ A.\ Rieffel \cite{Connes89, R98, R99, Rieffel00}; see also \cite{AguLat, ChristRieffel17,FLP24, Latremoliere16b, LoWu21}. When the weak-$*$ topology on the state space of a C*-algebra agrees with the topology induced by such a quantum metric, the state space is called \textit{a quantum compact metric space}. These spaces, introduced in their full generality by  M.\ A.\ Rieffel \cite{R98, R99, Rieffel00,Kerr03, Li06}, are thought of as a noncommutative analogue of the algebra of Lipschitz functions on a compact metric space, can also be obtained in other ways that do not require the use of spectral triples (particularly by F.\ Latrémolière \cite{AguLat, Latremoliere16b}), and have been thoroughly studied in the past two decades. 

Spectral triples  associated to 
groups endowed with a length function were introduced first by A.\ Connes \cite{Connes89}, and then shown to give rise to quantum compact space by  M.\ Christ and M.\ A.\ Rieffel in the bounded doubling case, see \cite{ChristRieffel17}.
Their  work was further generalized to the twisted setting by B.\ Long and W. \ Wu \cite{LongWuRes17, LoWu21}. This allowed the limit results of \cite{FaLaLaPa, FLP24};  for crossed products see \cite{AusKK25}. Very recently,  A.\ Austad and D.\ Kyed also studied the quantum metric structure arising from length functions on quantum groups \cite{Austad25}. While the work by A.\ Connes focuses on spectral triple aspects, M.\ A.\ Rieffel starts by considering metrics on ``state-like spaces'' of Banach algebras. 
Indeed, M.\ A.\ Rieffel's original definitions in \cite{R98} were already cast in a general normed space context and did not require a C*-algebra structure. This makes M.\ A.\ Rieffel's work a natural place to start when looking for analogues outside of the C*-algebra setting.

In very recent independent work whose topic is related to our own, C.\ Arhancet and  C.\ Kriegler have developed classical harmonic analysis in terms of noncommutative geometry, in particular  associating quantum compact metric spaces to some Markov semigroups of Fourier multipliers (imposing therefore a functional calculus requirement). Specifically, in  the papers \cite{Arhancet-Kriegler-22, Arhancet24, Arha24, Arha24CV, ArhaSob24}, (compact) Banach spectral triples  as well as associated  Banach Fredholm modules, $K$-theory and $K$-homology,  and index formulas have been introduced and studied. The case of the circle provides a foundational example  \cite{Arhancet-Kriegler-22, Arhancet24}.  In another direction, see \cite{GerMes25} for a study of  Monge--Kantorovi\v{c} metrics using Schatten ideals.\\

In this paper, we propose a generalization of the concept of spectral triples by replacing the Hilbert space framework with a more general $L^p$-space setting. This is done in Definition \ref{LpST}, which has already been used in  recent work \cite{lopes25}. We demonstrate that A.\ Connes' classical spectral triple example for group algebras extends to this $L^p$-based framework. Additionally, building on the work of E.\ Christensen and C.\ Ivan \cite{CI}, who constructed a Dirac operator for AF-algebras, we present an $L^p$-analogue for UHF-algebras of infinite tensor product type, as defined by N.\ C.\ Phillips in \cite{NCP13_2}. This construction leads to an $L^p$-spectral triple that induces a $p$-quantum compact metric on the state space of the algebra. 

We now lay out the structure of the paper. 
Section \ref{secPre} sets out our notational conventions and 
reminds the reader the basic theory of $L^p$-operator 
algebras. In particular $L^p$-analogues of well known constructions such as tensor products, reduced group algebras, and AF-algebras, are introduced. 
 In Section \ref{secDefiLp} we briefly go over classical definitions 
 of spectral triples and quantum compact metrics associated to the state space of a C*-algebra. 
All this is used as a motivation to introduce Definition \ref{LpST}, which gives 
an analogue of spectral triples in the $L^p$-setting. There is a recent notion of (compact) Banach spectral triples, using bisectorial bounded $H^\infty$ functional calculus, that also generalizes classical ones. 
The precise definition of a (compact) Banach spectral triple can be found in \cite[Definition 5.10]{Arhancet-Kriegler-22}.  In Remark \ref{CompDEFs} we provide details and comparisons between 
this and our definition for the $L^p$-case. Furthermore, 
we also present, in Subsection \ref{ex_Judy}, an instance of an $L^2$-spectral triple 
that satisfies Definition \ref{LpST}, but fails to admit a sectorial bounded $H^\infty$ functional calculus.
At the end of Section \ref{secDefiLp} we also briefly review
 M.\ A.\ Rieffel's general method for constructing quantum compact metrics. 

Our main results are presented in Sections \ref{Examples} and \ref{DOUHF}. In Subsection 
\ref{DOLpG}, for each $p \in [1, \infty)$, we mimic A.\ Connes' classical example by associating to a fixed proper 
length function $\LL$ on a countable discrete group $G$ a Dirac operator $D_\LL \colon C_c(G) \to \ell^p(G)$, by letting, for any $x \in G$, $\xi \in C_c(G)$,
\[
(D_\LL\xi)(x) \coloneqq \LL(x)\xi(x).
\]
 We then show that this 
operator does give rise to an $L^p$-spectral triple on $F^p_{\op{r}}(G)$,
the reduced $L^p$-group algebra of $G$: 

\begin{reptheorem}{GdisLpT}
Let $p \in [1, \infty)$, let $G$ be a countable discrete group, 
and let $\LL$ be a proper length function on $G$. Then, $(F^p_{\op{r}}(G), \ell^p(G), D_\LL)$ is an $L^p$-spectral triple
as in Definition \ref{LpST}.
\end{reptheorem} 

Next, when the length function on the group $G$ is of bounded doubling, 
we obtain a general $p$-version of the recent key result from \cite{ChristRieffel17}:

\begin{reptheorem}{thm_Gdb_}
Let $p \in [1, \infty)$, let $G$ be a countable discrete group, and let $\LL \colon G \to [0,\infty)$ be a 
length function of bounded doubling on $G$. Let $D=D_{\LL}$ be the associated 
multiplication operator. Then the seminorm $L_D(f) \coloneqq \| [D,\lambda_p(f)]\|$ 
metrizes the weak-* topology of $S(F_r^p(G))$.
\end{reptheorem}

In Section \ref{DOUHF}, motivated by the general C* AF-algebra setting provided in \cite{CI}, we prove a similar result for $L^p$ UHF-algebras.
To do so, we first develop a convenient equivalent definition 
for $L^p$ UHF-algebras of infinite tensor product type, which were originally defined 
in \cite{NCP13_2}. Roughly speaking, given a sequence $d \colon \Z_{\geq 0} \to \Z_{\geq 2}$, 
 we say that $A$ is an $L^p$ UHF-algebra
associated to the sequence $d$  
if it is isometrically isomorphic to 
\[
\cj{\bigcup_{n=0}^\infty A_n},
\]
where each $A_n$ is algebraically isomorphic to $ M_{d(1)\cdots d(n)}(\C)$, and the norm 
on $A_n$ comes from carefully representing it on an $L^p$-space built 
recursively, via spatial tensor products, from the representation of $A_{n-1}$. 
The precise definition we will use is contained in Definition \ref{Prob_leq_>} and the subsequent discussion. This naturally represents $A$ on $L^p(\mu_{\geq 0})$, where $\mu_{\geq 0}$ is a measure built from the representation of each $A_n$ (which we take to satisfy some 
finite dimensional assumptions, see \ref{A1}-\ref{A4} below). In turn, this allows us to define inclusion and projection maps within
the $L^p$-spaces on which $A$ and each $A_n$ are represented. These structural maps, together 
with any sequence of real numbers $(\alpha_{n})_{n=0}^\infty$ with $\alpha_0=0$ and $\alpha_n \to \infty$
give rise to a Dirac operator on $A$ given by 
\[
D_\alpha \coloneqq \sum_{n=1}^\infty \alpha_n(P_n-P_{n-1}),
\] 
where each $P_n \colon L^p(\mu_{\geq 0}) \to  L^p(\mu_{\geq 0})$ is an idempotent obtained by composing the projection and inclusion aforementioned. We then prove the following result

\begin{reptheorem}{UHFpST}
Let $p \in [1, \infty)$, let $A=A(d,\rho)$ be an $L^p$ UHF-algebra as defined in 
\eqref{UHFLP} and also satisfying assumptions \ref{A1}-\ref{A4}, and let $(\alpha_{n})_{n=0}^\infty$ with $\alpha_0=0$ and $\alpha_n \to \infty$. 
Then, $(A, L^p(\mu_{\geq 0}), D_\alpha)$ is an $L^p$-spectral triple 
as in Definition \ref{LpST}.
\end{reptheorem}

In section \ref{QMUHF}, we prove Theorem \ref{FinThm}, our final main result,  which says that one can always find a sequence $(\alpha_n)_{n=0}^\infty$ such that
Dirac operator $D_\alpha$ gives rise to a $p$-quantum compact metric 
on the state space the $L^p$ UHF-algebra from Theorem \ref{UHFpST}. 

Section \ref{Gp_UHF} ends the paper by presenting a unifying example.
Here, for each $n \in \Z_{\geq 2}$, we take $G_n$ to be the infinite direct sums of cylic groups of order $n$. We then put a length function of 
bounded doubling on $G_n$ so that $F_{\op{r}}^p(G_n)$ is now an $L^p$-UHF algebra that gives rise to a metric spectral triple by our results in Subsection \ref{SubSCRp}. This establishes instances in 
which both the results of Section \ref{Examples} and Section \ref{DOUHF} apply and is already of interest in the classical setting when $p=2$. 

\section*{Acknowledgments} This research was partly supported by: Simons Foundation Collaboration Grant \#523991 and   MPS-TSM-00007731   (C.F.), Simons Foundation Collaboration Grant \#1561094 (J.P.). The authors also thank Frédéric Latrémolière for helpful conversations on the topics covered in this work. We also thank Cédric Arhancet for letting us know about the useful references \cite{Arhancet-Kriegler-22, Arhancet24}.

\section{Preliminaries}\label{secPre}

Below we state our conventional notations that will be used throughout the 
paper. 
\begin{enumerate}
\item We begin by recalling the reader that a \textit{Banach space} is 
a complex normed vector space that is complete with the metric 
induced by the norm. A complex algebra $A$ is said to be a \textit{Banach algebra} if it is equipped with a submultiplicative norm that makes $A$ a Banach space. 
\item For any Banach space $(E, \|-\|_E)$ we write $\Li(E)$ to denote the Banach algebra 
of bounded operators on $E$. The algebra $\Li(E)$ comes equipped with the operator norm:
\[
\| a \| \coloneqq \sup_{\| \xi \|_E = 1 } \| a\xi\|_E.
\] 
Our convention from now on is that any norm without a subscript 
 refers to an appropriate operator norm. 
Also, $\mathcal{K}(E) \subseteq \Li(E)$ corresponds to 
the closed ideal of compact operators on $E$. 
\item If $(X, \mathfrak{M}, \mu)$ is a measure space and $p \in [1, \infty)$, we often
write $L^p(\mu)$ instead of $L^p(X, \mathfrak{M}, \mu)$, 
which is a Banach space when equipped with the usual $p$-norm:
\[
\| \xi \|_p  \coloneqq \|\xi\|_{L^p(\mu)} \coloneqq  \left(\int_X |\xi|^p d\mu\right)^{1/p}
\] 
When $\nu_X$ is the counting measure on any set $X$, 
we simply write $\ell^p(X)$ to refer to $L^p(X, 2^X, \nu_X)$. 
We adopt the usual notation and let $(\delta_x)_{x \in X}$ be the usual 
basis for $\ell^p(X)$, where for each $x \in X$, $\delta_x \colon X \to \C$ is
given by 
\[
\delta_x(y) \coloneqq \begin{cases} 1 & \text{ if }  x=y\\
0 & \text{ if } x\neq y
\end{cases}.
\]
Finally, when $d \in \Z_{\geq 1}$
we write $\ell_d^p$ instead of $\ell^p(\{1, \ldots, d\})$.  
\end{enumerate}

We also briefly recall for the reader some basic definitions and facts about 
linear, potentially unbounded, operators between Banach spaces. 

\begin{definition}
Let $E$ be a Banach space and let $\mathcal{D} \subseteq E$ be a (usually dense) subspace. 
For a linear map $a \colon \mathcal{D} \to E$ 
we define its \textit{resolvent set} $\op{res}(a)$ by 
\[
\op{res}(a) \coloneqq \{\lambda \in \C \colon  a-\lambda \op{Id}_E \colon \mathcal{D} \to E \text{
has a bounded everywhere-defined inverse}\}.
\]
To each $\lambda \in \op{res}(a)$ we associate the \textit{resolvent operator}
$R_\lambda(a) \in \Li(E)$ given by
\[
R_\lambda(a) \coloneqq ( a-\lambda \op{Id}_E)^{-1}.
\]
Finally, the \textit{spectrum of 
$a$} is then defined by $\sigma(a) \coloneqq \C \setminus \op{res}(a)$. 
\end{definition}

We now go over basic definitions and some known facts about $L^p$-operator algebras, their tensor products, 
and review two known constructions: reduced $L^p$-group algebras and $L^p$ AF-algebras.

\subsection{$L^p$-Operator Algebras}

\begin{definition}
Let $p \in [1, \infty)$, let  $(X, \mathfrak{M}, \mu)$ be a measure space, 
let $A$ be a Banach algebra, and let $\varphi \colon A \to \Li(L^p(\mu))$ be a representation of $A$ on 
$L^p(\mu)$, that is $\varphi$ is a continuous algebra homomorphism.
We say that $A$ is \textit{representable on $L^p(\mu)$}
whenever $\varphi$ is isometric, that $\varphi$ is \textit{$\sigma$-finite} if  
 $(X, \mathfrak{M}, \mu)$ is $\sigma$-finite, and that 
$A$ is \textit{$\sigma$-finitely representable} if 
it has a $\sigma$-finite isometric representation. 
\end{definition}

\begin{definition}\label{Lp}
Let $p \in [1, \infty)$. A Banach algebra $A$ is an 
\textit{$L^p$ operator algebra} if there is a measure space
$(X, \mathfrak{M}, \mu)$ such that $A$ is representable
on $L^p(\mu)$.  
\end{definition}

\subsection{Spatial Tensor Product}

For $p \in [1, \infty)$, there is a tensor product between an $L^p$-space and a general Banach space,
 called the 
\textit{spatial tensor product},  denoted by $\otimes_p$. 
We refer the reader to Section 7 of \cite{defflor1993} for complete details on this 
tensor product. The general setting starts 
with considering a measure space $(X_0, \mathfrak{M}_0, \mu_0)$ and a Banach space $E$.
We denote by $L^0(X_0,\mu_0, E)$ the vector space of measurable functions $X_0 \to E$ modulo functions that vanish a.e. $[\mu_0]$.
Then there is an isometric isomorphism 
\[
L^p(\mu_0) \otimes_p E \cong L^p(\mu_0, E)= \{ g \in L^0(X_0,\mu_0, E) 
\colon \textstyle{\int_{X_0}} \|g(x)\|_E^p \  d\mu_0(x) < \infty\},
\]
such that for any $\xi \in L^p(\mu_0)$ and $\eta \in E$, the elementary tensor $\xi \otimes \eta $ 
is sent to the function $x \mapsto \xi(x)\eta$.
In the special case where $E=L^p(\mu_1)$ for a measure space $(X_1,  \mathfrak{M}_1, \mu_1)$, then there is an isometric isomorphism
\begin{equation}\label{LpTensor}
L^p(\mu_0) \otimes_p L^p(\mu_1) \cong L^{p}(X_0 \times X_1, \mu_0 \times \mu_1),
\end{equation}
sending $\xi \otimes \eta$ to the function 
$(x_0,x_1) \mapsto \xi(x_0)\eta(x_1)$ for every $\xi \in L^p(\mu_0)$
and $\eta \in L^p(\mu_1)$. 
We describe the main properties of this 
tensor product between $L^p$-spaces below. 
The following is Theorem 2.16 in \cite{ncp2012AC}, 
except that we do not require the $\sigma$-finiteness assumption 
(see the proof in Theorem 1.1 in \cite{figiel1984}).
\begin{theorem}\label{SpatialTP}
Let $p \in [1, \infty)$ and for $j \in \{0,1\}$ let 
$(X_j, \mathfrak{M}_j, \mu_j)$, $(Y_j, \mathfrak{N}_j, \nu_j)$ be measure spaces.
\begin{enumerate}
\item Under the identification in \eqref{LpTensor}, 
$\op{span}\{ \xi \otimes \eta  \colon \xi \in L^p(\mu_0), \eta \in L^p(\mu_1)\}$ 
is a dense subset of $L^{p}(X_0 \times X_1, \mu_0 \times \mu_1)$.  \label{denseTensor_p}
\item $\| \xi \otimes \eta \|_p = \| \xi\|_p\|\eta\|_p$  for every 
$\xi \in L^p(\mu_0)$ and $\eta \in L^p(\mu_1)$. \label{LpTensorNorm}
\item Suppose that 
$a \in \Li(L^p(\mu_0), L^p(\nu_0))$
and $b \in \Li(L^p(\mu_1), L^p(\nu_1))$. 
Then there is a unique map 
$a\otimes b \in \Li(L^p(\mu_0 \times \mu_1),
 L^p( \nu_0 \times \nu_1))$ 
such that 
\[
(a \otimes b)(\xi \otimes \eta)=a\xi \otimes b\eta
\]
for every $\xi \in L^p(\mu_0)$ and $\eta \in L^p(\mu_1)$. 
Further, $\| a\otimes b\|=\|a\|\|b\|$. \label{LpTensorOp}
\item The tensor product of operators defined in \eqref{LpTensorOp} is associative, bilinear, 
and satisfies (when the domains are appropriate) 
$(a_1 \otimes b_1)(a_2 \otimes b_2) = a_1 a_2 \otimes b_1b_2$. \label{LpTensorOpProp}
\end{enumerate}
\end{theorem}

\subsection{Reduced $L^p$-Group Algebras}\label{LpG}
The reduced $L^p$-operator algebra of a group was introduced 
by C. Herz in \cite{CSH73}. Later,  N. C. Phillips introduced in \cite{NCP13} 
reduced (and full) $L^p$-crossed products, from which the $L^p$-group 
algebras come out as a particular examples. 
Below we briefly describe the construction for countable
discrete groups. A general definition 
for locally compact groups is Definition 3.1 in \cite{GarThi15}.
Let $G$ be a countable discrete group and let $p \in [1, \infty)$. 
For $f=(f(x))_{x \in G} \in \ell^1(G)$ and $\xi=(\xi(x) )_{x \in G} \in \ell^p(G)$ we define 
$f*\xi \colon G \to \C $ by letting for each $x \in G$, 
\[
(f*\xi)(x) \coloneqq \sum_{y \in G} f(y)\xi(y^{-1}x).
\]
It is standard to check that $f*\xi \in \ell^p(G)$ with $\|f*\xi\|_p \leq \|a\|_1\|\xi\|_p$. 
This gives rise to the left regular representation $\lambda_p: \ell^1(G) \to \mathcal{B}(\ell^p(G))$
 given, for each $f \in \ell^1(G)$, $\xi \in \ell^p(G)$, 
 by 
 \[
 \lambda_p(f)\xi \coloneqq f*\xi.
 \] 
 We define 
the \emph{reduced $L^p$ operator group algebra} $F^p_{\op{r}}(G)$ as 
the closure of $\lambda_p(\ell^1(G))$ in $\mathcal{B}(\ell^p(G))$. 

Notice that $F^p_{\op{r}}(G)$ is a unital Banach subalgebra of $\mathcal{B}(\ell^p(G))$, (making $F^p_{\op{r}}(G)$ an $L^p$-operator algebra in its own right) with unit $\lambda_p(\delta_{1_G})$. It is also clear that $F^p_{\op{r}}(G)$ is commutative when $G$ is abelian.

Observe also that $f \mapsto \lambda_p(f)$ is an injective unital 
contractive algebra homomorphism from $\ell^1(G)$ to $F^p_{\op{r}}(G)$,
whence $\ell^1(G)$ embeds as a dense subalgebra of $F^p_{\op{r}}(G)$.
That is, $\lambda_p(\ell^1(G))$ is a dense subalgebra of $F^p_{\op{r}}(G) \subseteq \mathcal{B}(\ell^p(G))$. In particular, this also gives that $\lambda_p(C_c(G))$ is a dense subalgebra of $F^p_{\op{r}}(G)$

\subsection{Spatial $L^p$ AF-Algebras}\label{LpAH}

Spatial $L^p$ AF-algebras are a particular type of $L^p$-operator algebras 
introduced (and classified) by N. C. Phillips and M. G. Viola in \cite{ncpmgv2017AF}. 
Here we recall their construction.

For $p \in [1,\infty)$ and $d \in \Z_{\geq 1}$ we set $M_{d}^p(\C) \coloneqq \Li(\ell_d^p)$ 
as the algebra of $d \times d$ complex matrices equipped with the $p$-operator
norm. According to \cite[Definition 7.9]{ncpmgv2017AF}, when $p \neq 2$, an $L^p$-operator algebra is \textit{spatial semisimple finite dimensional} 
if there are $N \in \Z_{\geq 1}$ and $d_1, \ldots, d_N \in \Z_{\geq 1}$
such that $A$ is isometrically isomorphic to 
\[
\bigoplus_{j=1}^N M_{d_j}^p(\C),
\]
where the direct sum is equipped with the usual max norm. 
If in addition $B$ is a unital $\sigma$-finitely representable $L^p$-operator 
algebra and $\varphi \colon A \to B$ is a (not necessarily unital) 
homomorphism, then we say $\varphi$ is \textit{spatial}
if, for $j=1, \ldots, N$, each 
$\varphi|_{M_{d_j}^p(\C)}$ is contractive and each $\varphi(1_{d_j})$ is a hermitian 
idempotent (see \cite[Definition 5.5]{ncpmgv2017AF}) where $1_{d_j} \in M_{d_j}^p(\C)$ is the $d_j \times d_j$ identity matrix. 

A spatial $L^p$ AF-algebra is defined in \cite[Definition 8.1]{ncpmgv2017AF}
as a Banach algebra that is isometrically isomorphic to the direct limit  
of a directed system 
\[
((A_n)_{n=0}^{\infty}, (\varphi_{m,n})_{0 \leq n \leq m})
\]
such that for each $n \geq 0$, the algebra $A_n$ is a spatial semisimple finite dimensional $L^p$-operator algebra and for all $m,n$ with $n\leq m$, the homomorphism $\varphi_{m,n} \colon A_n \to A_m$
is spatial. 

In this paper, we mainly work with a particular type of $L^p$ AF-algebras. 
These are the so called $L^p$ UHF-algebra of tensor product type as defined in  
\cite[Definition 3.9]{NCP13_2}. As in the C*-case, if $d \colon \Z_{\geq 0} \to \Z_{\geq 2}$
is a sequence, 
these $L^p$ UHF-algebras arise from direct limits of systems 
$((A_n)_{n=0}^{\infty}, (\varphi_{m,n})_{0 \leq n \leq m})$ where each $A_n$ 
is algebraically isomorphic to $M_{d(0)\cdots d(n)}(\C)$. 
Since we will require minor modifications 
to N.\ C.\ Phillips' construction, mainly a notational difference and extra assumptions, 
we go over this construction in detail in Subsection \ref{DOUHF_ST} below. 

\section{$L^p$-Spectral Triples and $p$-Quantum Compact Metric Spaces}\label{secDefiLp}

\subsection{$L^p$-Spectral Triples}
We begin this section by recalling  the classical 
definition of a spectral triple.

\begin{definition}\label{ClassicST}(\cite[Section 2]{Connes89})
A \textit{spectral triple} $(A, \Hi, D)$ consists of 
a unital C*-algebra $A$ faithfully represented on a 
Hilbert space $\Hi$ via $\varphi \colon A \to \Li(\Hi)$, 
and an unbounded selfadjoint operator $D \colon \op{dom}(D) \subseteq \Hi \to \Hi$, 
known as the \textit{Dirac operator}, 
such that 
\begin{enumerate}
\item $(I+D^2)^{-1} \in \mathcal{K}(\Hi)$, \label{cprRes}
\item $\{ a \in A \colon \| [D, \varphi(a)] \| < \infty\}$ is 
a norm dense $*$-subalgebra of $A$. \label{ComC*Dense}
\end{enumerate}

We remark that, since $D$ is selfadjoint, condition \eqref{cprRes} above 
is equivalent to asking that $D$ has compact resolvent, that is 
$(D-\lambda I)^{-1} \in \mathcal{K}(\Hi)$ for all $\lambda \in \C \setminus \sigma(D)$. 
\end{definition}

To generalize Definition \ref{ClassicST} to the $L^p$-setting, 
we note that there is not a good analogue of an unbounded 
selfadjoint operator on Banach spaces. 
Ideally we would like to replace selfadjoint with the well 
studied notion of hermitian operator acting 
on a Banach space $E$ 
where $a \colon \op{dom}(a) \subseteq E \to E$ is said to be \textit{hermitian} 
if its Lumer numerical range is a subset of $\R$ (see \cite{bonsall1970numerical, sinclair1971norm, giles1974numerical} and also \cite[Definition 5.5]{ncpmgv2017AF}). 
However, it was shown in \cite[Corollary 4]{giles1974numerical}
that any hermitian operator on a Banach space is actually bounded. 
Thus, there is no point in asking for a Dirac-like operator 
to be both unbounded and hermitian on a general $L^p$-space. 
To bypass this situation, we will require $D$ to satisfy an analogue of
condition \eqref{cprRes} in Definition \ref{ClassicST} and also will 
require $D$ to have compact resolvent. These two conditions 
are likely not equivalent for a general unbounded operator on a Banach space. 
Thus, asking both to hold simultaneously is, in some sense, a replacement 
for selfadjointness. 
This leads to the following definition of 
an $L^p$-spectral triple, presented below. 

\begin{definition}\label{LpST}
Let $p \in [1, \infty)$. An \textit{$L^p$-spectral triple} $(A, L^p(\mu), D)$ consists of 
a unital $L^p$-operator algebra $A$ faithfully represented on $L^p(\mu)$ via $\varphi \colon A \to \Li(L^p(\mu))$, 
and an unbounded operator $D \colon \op{dom}(D) \subseteq L^p(\mu) \to L^p(\mu)$, satisfying
\begin{enumerate}
\item $(I+D^2)^{-1} \in \mathcal{K}(L^p(\mu))$, \label{LpCpt}
\item $(D-\lambda I)^{-1} \in \mathcal{K}(L^p(\mu))$ for all $\lambda \in \C \setminus \sigma(D)$, \label{LpcprResol}
\item $\{ a \in A \colon \| [D, \varphi(a)] \| < \infty\}$ is 
a norm dense subalgebra of $A$. \label{ComLpDense}
\end{enumerate}
\end{definition}

We mention that Definition \ref{LpST}, even when $p=2$, differs from Definition \ref{ClassicST}.
This is intentional, as in Definition \ref{LpST} we also want to allow the notion of spectral triples 
over non-selfadjoint $L^2$-operator algebras, generalizing this concept to general 
operator algebras.

\begin{remark}\label{CompDEFs}
A related definition, that of a (compact) Banach spectral triple $(A, X, D)$,
 can be found in \cite[Definition 5.10]{Arhancet-Kriegler-22}.
  Roughly, in this definition we have that $A$ is an 
  algebra acting as bounded linear maps on a Banach space $X$ 
via a homomorphism $\pi \colon A \to \Li(X)$ and $D$ is a 
closed unbounded bisectorial operator on $X$ (see \cite[Definition 10.6.1.]{Hyt17}) with dense domain such 
that $D$ admits a certain bounded holomorphic functional
 calculus, such that $D^{-1}$ is compact on the closure 
 of its range, and such that $\{ a \in A \colon \| [D, \pi(a)] \| < \infty\}$ is 
a norm dense subalgebra of $A$. Moreover, in their definition the kernel of the Dirac operator can be infinite-dimensional.
In particular, we have checked, using \cite[Definition 10.6.1]{Hyt17}, that the Dirac operators 
on the $L^p$-spectral 
triples presented in Theorems \ref{GdisLpT} and \ref{UHFpST} 
are both $0$-bisectorial. On the other hand, it is
 not hard to check that the compact Banach spectral
  triple $(C^\infty(\mathbb{T}), L^p(\mathbb{T}), \frac{1}{i}\frac{d}{d\theta})$
   from \cite[Theorem 8.1]{Arhancet24} is also an $L^p$-spectral triple 
   in the sense of Definition \ref{LpST}. 
\end{remark}

In the following Subsection we give a sense of the differences between our definition for $L^p$-spectral triples and that of (compact) Banach triples. In particular we exhibit an $L^2$-spectral triple whose 
Dirac operator is bisectorial but does not satisfy a sectorial bounded $H^\infty$ functional calculus. 

\subsection{Comparing $L^p$ with Compact Banach Spectral Triples}\label{ex_Judy}
Recall that a conditional basis for a Banach space $E$ is a subset $\{\xi_n \colon n \in \N\} \subseteq E$
such that 
\begin{enumerate}
\item For every $\xi \in E$ there is a unique sequence $(\beta_n)_{n \in \N}$ in $\C$ 
such that $\xi = \sum_{n \in \N} \beta_n \xi_n.$
\item There is a sequence $(\theta_n)_{n \in \N}$ in $\{-1,1\}$ and a vector $\xi \in E$ 
such that $\xi = \sum_{n \in \N} \beta_n \xi_n$, but $\sum_{n \in \N} \theta_n\beta_n \xi_n$
does not converge. 
\end{enumerate}

It is known that every separable Hilbert spaces has a conditional basis. Below we exhibit a specific conditional basis  $\{\eta_n \colon n \in \N\}$ for $\ell^2(\N)$ by employing a minor modification of the well-known conditional basis described in \cite[Proposition 2.b.11.]{LinTza77} which is originally due to Babenko \cite{Babenko48}:

\begin{definition}\label{Def_condBasis}
Let $b \colon \ell^2(\Z)  \to \ell^2(\N)$ be the Hilbert space isomorphism defined by 
\[
b(\delta_n) \coloneqq \begin{cases} 
\delta_{2n} & n \geq 0\\
\delta_{-2n-1} & n <0 
\end{cases}.
\]
Fix $a\in (0,\frac{1}{2})$, set $K \coloneqq (\frac{2}{2a+1}\pi^{2a+1})^{1/2}$, and define the orthonormal functions $\{g_n \colon n \in \N\} \subseteq L^2([-\pi,\pi]),$ as follows:
	\[
	g_{2k}(x) \coloneqq \frac{1}{K}|x|^{a}e^{inx},\;g_{2k+1}(x) \coloneqq \frac{1}{K}|x|^ae^{-ikx},\;k \in \Z
	\]
Finally, let $\mathcal{F} \colon \ell^2(\Z) \to L^2([-\pi, \pi])$ be the Fourier transform, and 
for each $n \in \N$ we define
\[
\eta_{n} \coloneqq (b \circ \mathcal{F}^{-1})(g_n).
\]
Then $\{\eta_n \colon n \in \N\}$ is a conditional basis of $\ell^2(\N)$.
\end{definition}

Further, in \cite{Haase03, Venni93} it is shown that we can always associate an unbounded operator to a conditional basis on a Hilbert space, 
a processes we describe in the following definition. 
 
\begin{definition}
Consider the Hilbert space $\mathcal{H}=\ell^2(\Z_{\geq 0})$, with a fixed conditional basis $\{\xi_n \colon n\in \N\}$. For any sequence $\alpha= (\alpha_n)_{n\in \N}$ define on $\ell^2(\N)$ a, potentially unbounded, diagonal operator $T_\alpha$ as follows 
\begin{equation}\label{eq_Talpha}
\forall n\in \mathbb N \ : \ T_\alpha(\xi_n) \coloneqq \alpha_n \xi_n.
\end{equation}
with domain
\[
\op{dom}(T_\alpha) \coloneqq \left\{ \sum_{n \in \N} \beta_n \xi_n \in \ell^2(\N) \colon  \beta_n \in \C, \ \sum_{n \in \N} \beta_n\alpha_n \xi_n \text{ converges } \right\}.
\]
\end{definition}

The following proposition exhibits an $L^2$-spectral triple, that is  
satisfying Definition \ref{LpST} for $p=2$. However, we will see that the Dirac operator in the $L^2$-spectral triple from \ref{prop:count-Hilb} is bisectorial but
does not satisfy the sectorial bounded $H^\infty$ functional calculus. 

\begin{proposition}\label{prop:count-Hilb}	Consider $\mathcal{H}=\ell^2(\N)$ and let $T \in \mathcal{B}(\Hi)$ be  be the selfadjoint diagonal operator, with respect to a fixed orthonormal basis $\{\phi_n \colon n\in \N\}$, having $\frac{1}{n+1}$ in the $n^{th}$ diagonal entry for $n\geq 0,$ and $0's$ elsewhere.  Let $A$ be the commutative $C^*$-subalgebra of ${\mathcal B}(\ell^2(\N))$ generated by $\op{Id}$ and $T$.  Let $\alpha = (2^n)_{n\in\N}$ and $D \coloneqq T_\alpha$ as in Equation \eqref{eq_Talpha} using the conditional basis from Definition \ref{Def_condBasis}. 
Then $({A}, \ell^2(\N), D )$ is an $L^2$-spectral triple according to Definition \ref{LpST}.
\end{proposition}

\begin{proof}  
 It is evident that $T$ is a selfadjoint operator with spectrum equal to $ X \coloneqq \{0,1, \frac{1}{2}, \frac{1}{3}, \cdots\}$.
By the spectral theorem $A$ can be identified with $C(X),$ and polynomials in $T$ are dense in $A$. One can also describe $A$ as the diagonal matrices with respect to the orthonormal basis $\{\phi_n \colon n \in \N\}$ whose entries consist of convergent sequences of complex numbers  $(d_n)_{n\in\N}$ that form a convergent sequence, such that $\lim_{n\to \infty} d_n\;=\;d_0.$
	
Now find a change of basis matrix $S$ from the orthonormal basis $\{\phi_n \colon n\in\N\}$ to the conditional basis $\{\eta_n \colon n\in\N\}$.	
It might not be the case that $S$ is a bounded invertible operator on $\ell^2(\N)$ but for our purposes that will not matter, because we will only need $S$ defined on a dense subspace of $\ell^2(\N)$. Indeed, recall that $D$ is defined on $\op{dom}(T_\alpha)$, a dense subspace of $\ell^2(\mathbb N)$,  by 
\[
D\Big(\sum_{n=0}^{\infty}\beta_n\eta_n \Big) = \sum_{n=0}^{\infty}2^n\beta_n\eta_n.
\]
Hence, we can also calculate the inverse change of basis matrix $S^{-1},$ which again, like $D$, might only be defined on a dense subspace of $\ell^2(\N)$.
Now let $\tilde{D}$ be the diagonal matrix whose $n^{th}$ diagonal entry, with respect to the conditional basis,  is $2^n.$  Note on a dense subset of $\ell^2(\N)$ we have 
$$\tilde{D}\Big(\sum_{n=0}^{\infty}\beta_n\eta_n\Big)=\sum_{n=0}^{\infty}2^n\beta_n\eta_n.$$
It follows that, with respect to the orthonormal basis $\{\phi_n \colon n\in \N\}$,
we have the following matrix equation for our Dirac operator $D:$ 
$$D\;=\;S^{-1}\tilde{D}S.$$
We now note that a norm-dense subalgebra $A^\infty$ of the abelian $C^*$-algebra $A$ consists of 
elements of the form $\lambda \text{Id}+M,\;\lambda\in \mathbb C,$ where $M$ is a finite rank diagonal matrix of the form $M=(m_{i,j})$ where $m_{i,j}=0$ for $i\not=j,$ and $m_{n,n}=0$ for all $n> N.$
Now clearly $[D, \lambda \text{Id}]\;=\;0,$ and using the above matrix representation for $M,$ we see that  one calculates that the commutator $[D, M]\;=\;D\circ M-M\circ D$
 is a finite rank operator (since $m_{n,n}=0$ for $n>N$), thus bounded. It follows that if 
$N=\lambda\text{Id}+M,$ where $M$ is as described above,
$$[D, N]\;=\;[D,\lambda \text{Id}]+[D,M]$$
will also be a bounded operator.
We now have that  ${A}^{\infty}$ is norm  dense in ${A}.$ So we have found our desired dense subalgebra whose elements have bounded commutators with $D$. Finally we note that for all sequences $\alpha, \gamma,$ we have
	\[
	\forall n \in \mathbb N:\ [T_\alpha, T_\gamma^n ]=0,\ \ [T_\alpha, (T_\gamma^*)^n ]=0.
	\]
One easily uses these commutator identities to show that taking in particular 
$\alpha\;=(2^n)_{n \in \N},$ 
\[
\alpha'=\Big(\frac{1}{2^n-\lambda}\Big)_{n\in\N}, \alpha''=\Big(\frac{1}{1+2^{2n}}\Big)_{n\in\N}
\]
then
$$(D-\lambda\op{Id})^{-1}=T_{\alpha'} \ \text{ and} \ (\op{Id}+D^2)^{-1}=T_{\alpha''}.$$
	Since $\alpha'$ and $\alpha''$ both have limit equal to $0,$ it follows from  \cite[Page 50]{Haase03} that $(D-\lambda\text{Id})^{-1}$ and $(\text{Id}+D^2)^{-1}$ are both the limit in norm of finite rank operators, hence compact.
	\end{proof}

We have just shown that the triple $(A, \ell^2(\mathbb N), D)$ from Proposition \ref{prop:count-Hilb} is indeed an $L^2$-spectral triple. Furthermore, we know by \cite[Section 6, Proposition 21]{Beta19}  and \cite[Page 49]{Haase03} that the operator $D$ is sectorial and, since we are working with a sectorial operator of angle less than $\pi/2$ (our angle is $0$), the second stated consequence of \cite[Definition 10.6.1.]{Hyt17} gives that $D$ is indeed bisectorial.   However, it is also shown in \cite[Section 6, Proposition 21]{Beta19}  and \cite[Page 49]{Haase03} that $D$ does, at the very least, not satisfy the sectorial bounded $H^\infty$ functional calculus. Hence we conclude that $(A, \ell^2(\mathbb N), D)$ is an $L^2$- spectral triple in the sense of Definition \ref{LpST}, but might not be a (compact) Banach spectral triple in the sense of \cite[Definition 5.10]{Arhancet-Kriegler-22}. 

\subsection{$p$-Quantum Compact Metric Spaces}\label{QpM}

We now introduce a $p$-analogue for quantum compact metric spaces. We mention that (locally) compact quantum metric spaces associated to (compact) Banach spectral triples (see Remark \ref{CompDEFs}) were also considered in \cite[Section 5]{Arhancet-Kriegler-22}.

We start by recalling the general definition of a pseudometric: 

\begin{definition}
Let $X$ be any set, a function $\op{m} \colon X \times X \to \R_{\geq 0} \cup\{\infty\}$ is 
an \textit{extended pseudometric} on $X$ if  
\begin{enumerate}
\item $\op{m}(x,x)=0$ for all $x \in X$,
\item $\op{m}(x,y)=\op{m}(y,x)$ for all $x, y \in X$, 
\item $\op{m}(x,y) \leq \op{m}(x,z)+\op{m}(z,y)$ for all $x, y, z \in X$. 
\end{enumerate}
If $\op{m}(X \times X)=\R_{\geq 0}$ then $\op{m}$ is no longer extended. 
Further, whenever $\op{m}(x,y)=0$ implies that $x=y$, then $\op{m}$ is simply a 
usual (extended) metric. 
\end{definition}

We also need to work with general definition of states on a 
unital Banach algebra, which will be used in Subsection \ref{QMUHF} where we obtain a $p$-analogue of a quantum compact metric space
for the states on an $L^p$-operator algebra. The following is Definition 2.6 in \cite{blph2020}. 

\begin{definition}\label{States}
Let $A$ be a unital Banach algebra. We define its state 
space $S(A) \subseteq A'=\Li(A, \C)$ by 
\[
S(A) \coloneqq \{\omega \in  A'  \colon \|\omega\|= \omega(1_A)=1\}.
\]
\end{definition}
A well known fact is that when $A$ is a unital C*-algebra, then $S(A)$ is the usual definition of states in $A$.

In \cite{Connes89}, given a classical spectral triple $(A, \Hi, D)$ as in Definition \ref{ClassicST} above, A.\ Connes defined an extended pseudometric $\op{mk}_D$ on $S(A)$ 
by letting 
\begin{equation}\label{ExtPM}
\op{mk}_D(\omega, \psi) \coloneqq \sup\{|\omega(a)-\psi(a)| \colon \| [D,\varphi(a)] \| \leq 1\}. 
\end{equation}
If $\op{mk}_D$ is an actual metric whose induced topology on $S(A)$ 
coincides with the weak-$*$ topology inherited from $A'$, 
then we say that the triple $(A, \Hi, D)$ is \textit{metric} and the space $S(A)$ 
is known as a \textit{quantum compact metric space}. This name is justified as $S(A)$ is the noncommutative analogue of the spectrum of $C(X)$ for a compact topological space $X$. 
Note that the formula for $\op{mk}_D$ still makes sense when $(A, L^p(\mu), D)$ is an  $L^p$-spectral triple as in Definition \ref{LpST}.  In fact, in this case the formula for $\op{mk}_D$ in Equation \eqref{ExtPM} is still an extended pseudometric on $S(A)$ by the same proof 
used in \cite[Proposition 3]{Connes89}. 
Observe that we now have the following natural definition for $p$-quantum compact metric spaces: 
\begin{definition}\label{pQMetric}
Let $p \in [1, \infty)$. An $L^p$-spectral triple $(A, L^p(\mu), D)$, as in Definition \ref{LpST}, 
is said to be \textit{metric} when $\op{mk}_D$ gives rise to the  weak-$*$ topology on $S(A)$. In such case we also say that $\op{mk}_D$ is a \textit{$p$-quantum compact metric}.
\end{definition}

M.\ A.\ Rieffel gives in \cite[Theorem 1.8]{R98} necessary and 
sufficient conditions to define
a metric $\op{d}_S$ on a set $S \subseteq B'$ (where $B$ is simply a normed vector
space) such that the $\op{d}_S$-topology agrees with the weak-$*$ topology that $S$ inherits from $B'$. 
We will use the sufficient conditions in \cite[Theorem 1.8]{R98} to produce $p$-quantum compact metrics 
on $L^p$ UHF-algebras, so we recall, for the reader's benefit,  M.\ A.\ Rieffel's setting below. 
The data consists of
\begin{enumerate}[label=(\subscript{R}{{\arabic*}})]
\item \label{R1} A normed space $B$, with norm $\| - \|_B$, over either $\C$ or $\R$.
\item \label{R2}  A subspace $\mathcal{L}$ of $B$, not necessarily closed.
\item  \label{R3} A seminorm $L$ on $\mathcal{L}$.
\item  \label{R4} A continuous (for $\|-\|_B$) linear functional, $\varphi$, on
$\mathcal{N} = \{a \in  \mathcal{L} : L(a) = 0\}$
with $\| \varphi \|= 1$. In particular, we require $\mathcal{N} \neq \{0\}$.
\item \label{R5} Let $B'=\Li(B,\C)$ and set
\[
S \coloneqq \{\omega \in  B' : \omega = \varphi \text{ on $\mathcal{N}$, and }\|\omega\|= 1\}.
\]
Thus $S$ is a norm-closed, bounded, convex subset of $B'$, and so is weak-$*$ compact.
We ask that $\mathcal{L}$ separate the points of $S$.
This means that given distinct $\omega, \psi \in  S$ there is a $b \in \mathcal{L}$ such that $\omega(b) \neq  \psi(b)$. 
\end{enumerate}
With the data in \ref{R1}-\ref{R5}, we define a function $\op{d}_S \colon S \times S \to \R_{\geq 0} \cup \infty$, by the formula 
\[
\op{d}_S(\omega, \psi) \coloneqq  \sup\{ |\omega(b) - \psi(b)|  \colon  b \in \mathcal{L},\ L(b) \leq 1 \}.
\]
The assumed data guarantees that $\op{d}_S$ is an extended metric on $S$ that separates the points of $S$. We will refer
to the topology on $S$ defined by $\op{d}_S$ as the `$\op{d}_S$-topology'. Further, $L$ descends 
to an actual norm on the quotient space $\mathcal{L}/\mathcal{N}$, but this quotient space is also equipped with the quotient norm induced by $\| - \|_B$. The latter norm 
is the one we care about, which will henceforward denote as $\| - \|_{B/\mathcal{N}}$. The following is the forward direction of Theorem 1.8 in \cite{R98}:
\begin{theorem}\label{RSuffT}
Let $B$, $\mathcal{L}$, $\mathcal{N}$, and $S$ be as in \ref{R1}-\ref{R5}. Define $\mathcal{L}_1 \subseteq \mathcal{L}$ by
\[
\mathcal{L}_1 \coloneqq \{b \in \mathcal{L} \colon L(a) \leq 1\}.
\]
If the image of $\mathcal{L}_1$ in $\mathcal{L}/\mathcal{N}$
is totally bounded for $\| - \|_{B/\mathcal{N}}$, then the $\op{d}_S$-topology on $S$ agrees with the weak-$*$ topology.
\end{theorem}

We now have all the necessary background to present our main results, which consist of exhibiting $L^p$-spectral triples that are metric in the sense of Definition \ref{pQMetric}. We present two basic constructions, one for $L^p$-group algebras in Section \ref{Examples} and the other for UHF $L^p$-algebras in Section \ref{DOUHF}. 

\section{Quantum Metrics on $L^p$-Group Algebras}\label{Examples}

 This section 
focuses only with the $L^p$-group algebra case. 
In Subsection \ref{DOLpG} we present $L^p$-spectral triples associated with $L^p$-group algebras and length functions. In Subsection \ref{SubSCRp} we show that when the group has a length function of bounded doubling, then the associated $L^p$-spectral triple is metric.

\subsection{Dirac Operators on $F_\op{r}^p(G)$ via Length Functions.}\label{DOLpG}

We first recall the definition of length function 
(see \cite{ChristRieffel17, Connes89, LoWu21}).

\begin{definition}\label{LengthF}
A \emph{length function} on a group $G$ is a map $\LL \colon G \to \R_{\geq 0}$
satisfying the following properties 
\begin{enumerate}
\item $\LL(x)=0$ if and only if $x$ is the identity of $G$,
\item $\LL(xy) \leq \LL(x) + \LL(y)$ for all $x,y \in G$,
\item $\LL(x^{-1}) = \LL(x)$ for all $x \in G$.
\end{enumerate}
We say $\LL$ is \emph{proper}, if in addition  
\begin{enumerate}  \setcounter{enumi}{3}
\item for any $r \in \R_{\geq 0}$, $B_{\LL}(r) \coloneqq \LL^{-1}([0, r])$ is a finite subset of $G$.
\end{enumerate}
Furthermore, a proper length function $\LL$ on $G$ 
is said to be 
\begin{enumerate}
  \setcounter{enumi}{4}
  \item of \emph{bounded doubling} if there is a constant $C_{\LL}<\infty$ such that $\op{card}(B_{\LL}(2r)) \leq C_{\LL}\op{card}(B_{\LL}(r))$ for all $r \in \R_{\geq 1}$. \label{bddDoub}
  \item of \emph{bounded $t$-dilation} for a fixed $t \in \R_{\geq 1}$ if there is a constant $K_{\LL}<\infty$ such that $\op{card}(B_{\LL}(tr)) \leq K_{\LL}\op{card}(B_{\LL}(r))$ for all $r \in \R_{\geq 1}$. 
\end{enumerate}

\end{definition}
Let $p \in [1, \infty)$ and let $G$ be a countable discrete group.
For a given proper length function $\LL$ on $G$, we define $D_\LL \colon C_c(G) \to \ell^p(G)$ by letting, 
for any $\xi \in C_c(G)$ and $x \in G$, 
\begin{equation}\label{LpGDirac}
(D_\LL\xi)(x) \coloneqq \LL(x)\xi(x).
\end{equation}

For $p \in [1, \infty)$ and a countable discrete group $G$, recall from Subsection \ref{LpG} that $F_\op{r}^p(G)$ is defined as the closure of $\lambda_p(\ell^1(G))$ in $\Li(\ell^p(G))$. Thus, $F_\op{r}^p(G)$ is by definition 
an $L^p$-operator algebra, see Definition \ref{Lp}, concretely represented on $\ell^p(G)$ by its canonical inclusion map. In the following theorem  we show that $D_\LL$ is indeed a $p$-analogue of a Dirac operator, now acting on $\ell^p(G)$,  associated to 
any proper length function on a discrete group $G$, thus generalizing 
A.\ Connes' classical construction from \cite{Connes89}.

\begin{theorem}\label{GdisLpT}
Let $p \in [1, \infty)$, let $G$ be a countable discrete group, 
let $\LL$ be a proper length function on $G$ and let $D \coloneqq D_\LL$ 
as in Equation \eqref{LpGDirac}. Then, $(F^p_{\op{r}}(G), \ell^p(G), D)$ is an $L^p$-spectral triple
as in Definition \ref{LpST}. That is, 
\begin{enumerate}
\item $(I+D^2)^{-1} \in \mathcal{K}(\ell^p(G))$, \label{I+Dcpt}
\item $(D-\lambda I)^{-1} \in \mathcal{K}(\ell^p(G))$ for all $\lambda \in \C \setminus \sigma(D)$. , \label{lpGResol}
\item $\{ a \in F^p_{\op{r}}(G) \colon \| [D, a] \| < \infty\}$ is 
a norm dense subalgebra of $F^p_{\op{r}}(G)$. \label{CommDense}
\end{enumerate}
\end{theorem}
\begin{proof}
A straightforward computation shows that $\|(I+D^2)\xi\|_p \geq \|\xi\|_p$ 
for any $\xi \in C_c(G)$. Furthermore, $\ker(I+D^2)=\{0\}$ and 
\[
\cj{(I+D^2)(C_c(G))}^{\|-\|_p}=\cj{C_c(G)}^{\|-\|_p}=\ell^p(G).
\] 
Hence $I+D^2$ is continuously invertible on its range, which yields $(I+D^2)^{-1} \in \mathcal{B}(\ell^p(G))$. 

Let $\Lambda = \{ F \subseteq G \colon F \text{ is finite }\}$ ordered by inclusion. For each $F \in \Lambda$ 
we define $K_F : \ell^p(G) \to \ell^p(G)$ by letting, for each $\xi \in \ell^p(G)$ and $x \in G$, 
\[
(K_F\xi)(x) \coloneqq \frac{1}{1+\LL(x)^2}\xi(x)\chi_F(x),
\]
where $\chi_F$ is the indicator function for the set $F$. For each $F \in \Lambda$, direct computations show that $K_F$ is linear, $\|K_F\|\leq 1$, and that $K_F$ has finite rank. Moreover, since $\LL$ is proper, it 
also follows that $\lim_{F \in \Lambda} \| (I+D^2)^{-1}-K_F\|=0$, whence $(I+D^2)^{-1} \in \mathcal{K}(\ell^p(G))$, proving part \eqref{I+Dcpt}. Part \eqref{lpGResol} follows analogously 
after observing that $\sigma(D)=\LL(G)$, so that in this case the finite rank 
operator  $J_F \colon \ell^p(G) \to \ell^p(G)$ approximating $(D-\lambda I)^{-1}$ when $\lambda \not\in \sigma(D)$
is defined by
\[
(J_F\xi)(x) \coloneqq \frac{1}{\LL(x)-\lambda}\xi(x)\chi_F(x).
\]
Finally, we claim that 
\[
\lambda_p(C_c(G)) \subseteq \{ a \in F^p_{\op{r}}(G) \colon \| [D,a] \| < \infty\},
\]
where as usual $\lambda_p \colon \ell^1(G) \to \mathcal{B}(\ell^p(G))$ is the left regular representation 
(see Subsection \ref{LpG}). 
To prove such claim, take any $f \in C_c(G)$ and let $\op{supp}(f)=\{ x \in G \colon f(x) \neq 0\}$. Since $f \in C_c(G)$, 
we have $\op{supp}(f) \in \Lambda$. Then, a direct computation shows
\[
\| [D, \lambda_p(f)] \|^p = \sup_{\| \xi \|_p =1} \sum_{x \in G} \left| D(f*\xi)(x)-(f*D\xi)(x)\right|^p \leq \|f\|_1^p  \sum_{x \in \op{supp}(f)} \LL(x)^p < \infty,
\]
proving the claim. Thus, we have shown that $\{ a \in F^p_{\op{r}}(G) \colon \| [D,a] \| < \infty\}$ contains $\lambda_p(C_c(G))$ which is
already a norm dense subalgebra of $F^p_{\op{r}}(G)$, so part 
\eqref{CommDense} follows. 
\end{proof}

\subsection{$p$-Quantum Compact Metric for $F^p_{\op{r}}(G)$}\label{SubSCRp}

In this subsection we show that the main result in  \cite{ChristRieffel17}, 
still holds for any $p \in [1, \infty)$ with some modifications. 
Indeed, new techniques are needed to replace the $*$-operator, which is missing in the Banach setting: in our case the inner product will be replaced by dual $L^p$-pairing and using the H\"older inequality.

Henceforth, $\LL \colon G \to [0, \infty)$ will be a length function on $G$ as in Definition 
\ref{LengthF}, $D \coloneqq D_{\LL}$ as in Equation \eqref{LpGDirac}, and 
for any $f \in \ell^1(G)$ we let
\[
L_{D}(f) \coloneqq \| [D , \lambda_p(f)]\|_{\Li(\ell^p(G))}.
\]
Also for any $h \in \ell^\infty(G)$ we let $M_h \colon \ell^p(G) \to \ell^p(G)$ 
be the multiplication operator defined, for each $\xi \in \ell^p(G)$ and $x \in G$, by 
\[
(M_h\xi)(x)\coloneqq h(x) \xi(x).
\]
 In particular, for each $r \geq 0$ we adopt the notation
$M_r \coloneqq M_{\chi_{B_\LL(r)}}$. Then for any $k \in \Z_{\geq 0}$ and $\xi \in \ell^p(G)$ 
we have 
\[
\| D^kM_r \xi \|_{p} = \Bigg(\sum_{x \in B_\LL(r)}|\LL^k(x)\xi(x)|^p \Bigg)^{1/p}\leq r^{k}\|\xi\|_p,
\]
from where it follows that $\| D^kM_r\|_{\Li(\ell^p(G))} \leq r^k$. 
A similar computation shows that $\| D^{-k}(I-M_r)\|_{\Li(\ell^p(G))} \leq r^{-k}$.
Thus, by the same arguments used to prove Proposition 2.1 in \cite{ChristRieffel17}, 
we now obtain the following proposition.
\begin{proposition}\label{I-MLM}
Let $p \in [1, \infty)$, $f \in C_c(G)$ and $r,s \in \R$ with $s>r\geq 0$. Then 
\[
(I-M_s)\lambda_p(f) M_r = \sum_{k=0}^\infty D^{-(k+1)}(I-M_s)[D,\lambda_p(f)]D^kM_r,
\]
and therefore $\| (I-M_s)\lambda_p(f) M_r \|_{\Li(\ell^p(G))} \leq \frac{L_D(f)}{s-r}$.
\end{proposition}

Next we show an analogue of Proposition  3.1 in \cite{ChristRieffel17}. 
The statement has the obvious modifications and to prove it we will need some known facts about 
Hölder duality that we list below. 

For $p \in [1,\infty)$ we fix $q$ for its Hölder conjugate (that is $\frac{1}{p}+\frac{1}{q}=1$). 
We have the usual pairing $\ell^q(G) \times \ell^p(G) \to \C$  given by 
\[
(\eta, \xi ) \mapsto ( \eta \mid \xi ) \coloneqq \sum_{g \in G} \eta(g) \xi(g).
\]
We will often use the following well known facts: 
\begin{enumerate}
\item $\|\xi\|_p = \sup_{\| \eta\|_q=1} | ( \eta \mid \xi ) |,$ \label{dual Norm}
\item $\xi_n \to \xi$ weakly in $\ell^p(G)$ if and only if $( \eta \mid \xi_n )  \to ( \eta \mid \xi )$ for all $\eta \in \ell^q(G)$
\item $a_n \to a$ in the WOT of $\Li(\ell^p(G))$ if and only if $( \eta \mid a_n\xi ) \to ( \eta \mid a\xi ) $ for all $\xi \in \ell^p(G)$, $\eta \in \ell^q(G)$,
\item $\xi_1 = \xi_2 \in \ell^p(G)$ if and only if $( \eta \mid \xi_1 ) =( \eta \mid \xi_2 ) $ for 
all $\eta \in C_c(G)$. 
\end{enumerate}

Next, for each $y \in G$ we use $\rho_y \colon \ell^p(G) \to \ell^p(G)$ to denote right translation (that is $\rho_y(\xi)(x) \coloneqq \xi(xy^{-1})$ for all $x \in G$) and for any $h \in C_c(G)$ we define $\tilde{h}, h^* \in C_c(G)$ by letting for each $x \in G$, 
\[
\tilde{h}(x) \coloneqq h(x^{-1}) \ \text{ and } h^*(x)\coloneqq \cj{h(x^{-1})}.
\] 
We are now ready to state and prove the analogue of Proposition  3.1 in \cite{ChristRieffel17}: 

\begin{proposition}\label{pLambdaanalogue}
Let $p \in [1, \infty)$ and let $f,h,k \in C_c(G)$. 
Then
\[
\lambda_p((h^* * k)f)= \sum_{z \in G} \rho_{z^{-1}} M_{h^*} \lambda_p(f) M_{\tilde{k}} \rho_z,
\]
where the convergence is in the WOT of $\Li(\ell^p(G))$. Furthermore, 
\[
\| \lambda_p((h^* * k)f)\|_{\Li(\ell^p(G))} \leq \| \lambda_p(f)\|_{\Li(\ell^p(G))} \| h \|_q \| k \|_p.
\]
\end{proposition}
\begin{proof}
Observe that for any $\xi \in \ell^p(G)$, $\eta \in \ell^q(G)$ and $z \in G$ we have 
\[
(\eta \mid M_{\tilde{k}} \rho_z\xi ) = \sum_{x \in G} \eta(x) k(x^{-1})\xi(xz^{-1}) = 
\sum_{y \in G} \eta(yz) \tilde{k}(yz)\xi(y) = (\rho_{z^{-1}}M_{\tilde{k}}\eta \mid \xi)
\]
Thus, a computation analogous to the one done by M.\ Christ  and M.\ A.\ Rieffel (replacing $\cj{\eta}$ by $\eta$) shows that 
\[
(\eta \mid \lambda_p((h^* * k)f) \xi) = \Bigg( \eta { \ }\Bigg|{ \ } \sum_{z \in G} \rho_{z^{-1}} M_{h^*} \lambda_p(f) M_{\tilde{k}} \rho_z\xi \Bigg).
\]
The norm bound does have a slightly different proof that needs Hölder duality. 
To see this, take $\xi \in \ell^p(G)$ and $\eta \in \ell^p(G)$.
Define $g_{\xi, k} \colon G \to \R$ and $g_{\eta, h} \colon  G \to \R$
by 
\[
g_{\xi, k}(z)\coloneqq \|  M_{\tilde{k}} \rho_z \xi\|_p, \  g_{\eta, h}(z) \coloneqq \| M_{\tilde{h}} \rho_z \eta\|_q. 
\]
We claim that $g_{\xi, k} \in \ell^p(G)$ and $g_{\eta, h} \in \ell^q(G)$. 
Indeed, 
\[
\|g_{\xi, k}\|_p^p=\sum_{z \in G} \|M_{\tilde{k}} \rho_z \xi\|_p^p = \sum_{z \in G} \sum_{x \in G} | k(x^{-1})\xi(xz^{-1})|^p
=\| \xi\|_p^p \| k \|_p^p. 
\]
An analogous computation shows that $\| g_{\eta, h}  \|_q^q = \| \eta\|_q^q \|h\|_q^q$, so our claim now follows. 
Then, two applications of the Hölder's inequality 
now give 
 \begin{align*}
| (\eta \mid \lambda_{(h^* * k)f} \xi) | & \leq \sum_{z \in G} |(M_{\tilde{h}} \rho_z \eta \mid \lambda_p(f) M_{\tilde{k}} \rho_z \xi)| \\
& \leq \| \lambda_p(f)\|_{\Li(\ell^p(G))} \|g_{\xi, k}\|_p\| g_{\eta, h}  \|_q  \\
& = \| \lambda_p(f)\|_{\Li(\ell^p(G))} \|\xi\|_p \|k\|_p\|\eta\|_q\| h  \|_q.
 \end{align*}
 Hence, this combined with \eqref{dual Norm} gives $\|  \lambda_p((h^* * k)f)\xi \|_p \leq \| \lambda_p(f)\|_{\Li(\ell^p(G))} \|\xi\|_p \|k\|_p\| h  \|_q$, from which it now follows at once that 
 \[
 \| \lambda_p((h^* * k)f)\|_{\Li(\ell^p(G))} \leq \| \lambda_p(f)\|_{\Li(\ell^p(G))} \| h \|_q \| k \|_p,
 \]
 as wanted. 
\end{proof}

We now obtain the $p$-version of Proposition 3.2 in \cite{ChristRieffel17}: 

\begin{proposition}
Let $p \in [1, \infty)$ and let $f,h,k \in C_c(G)$. 
Then
\[
L_D ( \  (h^**k)f \ ) \leq \| h \|_q \| k \|_p L(f).
\]
\end{proposition}
\begin{proof}
This follows again by careful applications of Hölder's inequality as in the proof of Proposition \ref{pLambdaanalogue} above.
\end{proof}

Following Definition 3.3 of \cite{ChristRieffel17} we define a seminorm $J_D$ on $C_c(G)$ 
by 
\[
J_D(f) \coloneqq  \sup\{ r \| (I - M_{2r})\lambda_p(f) M_r \|_{\Li(\ell^p(G))} \colon  r \geq 0 \}.
\]
Then, Proposition \ref{I-MLM} shows that $J_D \leq L_D$ on $C_c(G)$. Our norm estimates 
paralleling Propositions 3.5 and 3.6 in \cite{ChristRieffel17} are 
respectively given by 
\begin{align*}
\| \lambda_p((h^* * k)f)\| & \leq r^{-1} \| h \|_q \| k \|_p J_D(f)\\ 
J_D( \  (h^**k)f \  ) & \leq \| h \|_q \| k \|_p J_D(f). 
\end{align*}
Similarly, the analogue of  Corollary 3.7 in \cite{ChristRieffel17}
now says that if $r>0$ and $E \subseteq B_\LL(r)$, $F \subseteq G\setminus B_\LL(2r)$
then putting $k=\chi_E$ and $h=\chi_{F}$ yields,  
\begin{align*}
\| \lambda_p((h^* * k)f)\| & \leq r^{-1} \op{card}(F)^{1/q}  \op{card}(E)^{1/p} J_D(f)\\ 
J_D( \  (h^**k)f \  ) & \leq \op{card}(F)^{1/q}  \op{card}(E)^{1/p} J_D(f). 
\end{align*}

Let $A(s,t) = B_\LL(t) \setminus B_\LL(s)$ for $t>s>0$. If $s>2r>0$, 
$k=\op{card}(B_\LL(r))^{-1}\chi_{B_\LL(r)}$, $h=\chi_{A(s,t)}$, and $g=h^**k$, 
the $p$-version of  Corollary 3.9 in \cite{ChristRieffel17} is that 
for any $f \in C_c(G)$
\[
\| \lambda_p(gf)\| \leq r^{-1} (\op{card}(B_\LL(r))^{-1}\op{card}(B_\LL(t)))^{1/q}J_D(f)
\]
and
\[
J_D(gf) \leq  (\op{card}(B_\LL(r))^{-1}\op{card}(B_\LL(t)))^{1/q}J_D(f).
\]
Therefore, imposing the assumption that the length $\mathbb{L}$ on group $G$ is of bounded doubling (see part \eqref{bddDoub} in Definition \ref{LengthF}), the results in \cite[Section 4]{ChristRieffel17} are almost identical in this case, but now using the alternative constant $C_1'=C_{\LL}^{K/q}$. It now follows that 
 \[
 B_J=\{\lambda_p(f) \colon f \in C_c(G), f(1_G)=0, J_D(f)\leq 1\}
 \]
 is totally bounded with respect to the norm in $F_r^p(G)$. Altogether, we have shown the following main result:

\begin{theorem}\label{thm_Gdb_}
Let $p \in [1, \infty)$, let $G$ be a countable discrete group, and let $\LL \colon G \to [0,\infty)$ be a 
length function of bounded doubling on $G$. Let $D=D_{\LL}$ be the associated 
multiplication operator. Then the seminorm $L_D(f) \coloneqq \| [D,\lambda_p(f)]\|$ 
metrizes the weak-* topology of $S(F_r^p(G))$.
\end{theorem}

In terms of Definition \ref{pQMetric} we have shown the following result:
\begin{corollary}
Let $p \in [1, \infty)$, let $G$ be a countable discrete group, and let $\LL \colon G \to [0,\infty)$ be a 
length function of bounded doubling on $G$. 
Then $(F_r^p(G), \ell^p(G), D_{\LL})$ is a metric $L^p$-spectral triple. 
\end{corollary}

\section{Quantum Metrics on $L^p$ UHF-Algebras}\label{DOUHF}

In this section we define $L^p$-spectral triples for certain $L^p$ UHF-algebras and, in Subsection \ref{QMUHF} we show that we can actually construct $p$-quantum 
compact metrics on them. The results are motivated largely by the well-known construction by Christensen-Ivan of spectral triples, and in particular Dirac operators, on 
AF C*-algebras (see \cite{CI}). 

\subsection{Dirac Operator on $L^p$ UHF-Algebras.}\label{DOUHF_ST}

Here we construct an $L^p$-analogue of the Christensen-Ivan triple for a general 
$L^p$ UHF-algebra of tensor product type as in defined in  
\cite[Definition 3.9]{NCP13_2}. 

Fix $p \in [1, \infty)$. First of all we recall from \cite{NCP13_2} the construction of $A(d, \rho)$, the $L^p$ UHF-algebra of tensor product type associated to a sequence $d \colon \Z_{\geq 0} \to \Z_{\geq 2}$ and to $\rho = (\rho_j)_{j=0}^\infty$, a sequence of unital representations $\rho_j \colon M_{d(j)}(\C) \to \mathcal{B}(L^p(\mu_j))$ on 
probability spaces $(X_j, \mathfrak{M}_j, \mu_j)$. The following construction is taken from Example 3.8 
in \cite{NCP13_2}, but we are using a slightly different notation in 
order to compare it to the C*-case from \cite{CI}. 

\begin{definition}\label{Prob_leq_>}
Let $(X_j, \mathfrak{M}_j, \mu_j)_{j=0}^\infty$
be a sequence of probability spaces. We define
the following product probability spaces:
\begin{enumerate}
\item  \label{FinProd} For each $m \in \Z_{\geq 0}$
\[
(X_{\leq m}, \mathfrak{M}_{\leq m}, \mu_{\leq m}) \coloneqq \prod_{j=0}^m (X_j, \mathfrak{M}_j, \mu_j).
\]
\item \label{InfProd} For each $n \in \Z_{\geq 0}$ we define 
\[
(X_{\geq n}, \mathfrak{M}_{\geq n}, \mu_{\geq n}) \coloneqq  \prod_{j=n}^\infty (X_j, \mathfrak{M}_j, \mu_j). 
\]
\item \label{IntervalProd} For $n \in \Z_{\geq -1}$ and $m \in \Z_{\geq n}$
we define 
\[
(X_{(n,m]}, \mathfrak{M}_{(n,m]}, \mu_{ (n,m]}) \coloneqq  \prod_{j=n+1}^m (X_j, \mathfrak{M}_j, \mu_j).
\]
Note that for every 
$n \in \Z_{\geq 0}$, the space $(X_{(n,n]}, \mathfrak{M}_{(n,n]}, \mu_{ (n,n]})$ is simply an empty product of probability spaces, which is understood as a one point space with counting measure.
\end{enumerate}
\end{definition}
Observe that for $p \in [1, \infty)$ and any $n \in \Z_{\geq 0}$ we have an obvious isometric 
identification 
as in Equation \eqref{LpTensor}
\begin{equation}\label{tensorDecomp}
L^p(\mu_{\geq 0}) = L^p(\mu_{\leq n}) \otimes_p  L^p(\mu_{\geq n+1}).
\end{equation}
We will use $\op{Id}_{\leq n}$ and $\op{Id}_{\geq n}$ to respectively denote the identity operators on $L^p(\mu_{\leq n})$ and $L^p(\mu_{\geq  n})$. 
Next, for each $n \in \Z_{\geq 0}$, we define
\[
M_{n} \coloneqq  \bigotimes_{j=0}^n M_{d(j)}(\C).
\]
Observe that $M_n$ is algebraically isomorphic to $M_{d(0)d(1)\cdots d(n)}(\C)$. 
Now, the sequence $\rho=(\rho_j)_{j=0}^\infty$ together with part \eqref{LpTensorOp} in Theorem \ref{SpatialTP} induces 
a unique unital representation $\rho_{\leq n} \colon M_n \to \mathcal{B}(L^p(\mu_{\leq n}))$ via 
\[
\rho_{\leq n} (a_0 \otimes a_1 \otimes \cdots \otimes a_n) \coloneqq  \rho_0(a_0)\otimes \rho_1(a_1) \otimes \cdots \otimes \rho_n(a_n),
\]
where $a_j \in M_{d(j)(\C)}$ for each $j \in \{0, \ldots, n\}$. 
This allows us to equip each $M_n$ with the norm $\| a \| = \| \rho_{\leq n}(a)  \|$, 
making $M_n$ an $L^p$-operator algebra isometrically represented on $L^p(\mu_{\leq n})$. 
Further, for any $n,m \in \Z_{\geq 0}$ with $n<m$ we have obvious isometric
homomorphisms $\varphi_{m,n} \colon M_n \to M_m$ obtained by filling in tensor 
factors with corresponding identity matrices. To be precise we have 
\[
\varphi_{m,n}(a_0 \otimes \ldots \otimes a_n) = a_0 \otimes \ldots \otimes a_n \otimes 1_{d(n+1)} \otimes \cdots \otimes 1_{d(m)},
\]
where for each $j \in \Z_{\geq 0}$ the element $1_{d(j)} \in M_{d(j)}(\C)$ is the $d(j) \times d(j)$ identity matrix. This gives rise to a unital isometric representation 
of $M_{n}$ on $L^p(\mu_{\leq m})$ by defining $\rho_{m,n} \colon M_n \to \mathcal{B}(L^p(\mu_{\leq m}))$ by \[
\rho_{m,n} \coloneqq \rho_{\leq m} \circ \varphi_{m,n}.
\] 
Next, with the identification in Equation \eqref{tensorDecomp}, we 
get a unital isometric representation of $M_n$ on $L^p(\mu_{\geq 0})$, 
denoted as $\rho_{\infty, n} \colon M_n \to \mathcal{B}(L^p(\mu_{\geq 0}))$, 
by letting 
\[
\rho_{\infty, n}(a) \coloneqq  \rho_{\leq n}(a) \otimes \op{Id}_{\geq n+1}.
\]
Finally, for each $n \in \Z_{\geq 0}$ we define \[
A_n \coloneqq \rho_{\infty, n}(M_n) \subseteq \mathcal{B}(L^p(\mu_{\geq 0})), 
\]
and then put 
\begin{equation}\label{UHFLP}
 A(d,\rho) \coloneqq  \cj{ \bigcup_{n=0}^\infty A_n }  \subseteq \mathcal{B}(L^p(\mu_{\geq 0})).
\end{equation}

\begin{remark}\label{ConcrRepnUHF}
Observe that for each $n \in \Z_{\geq 0}$, $A_n$ is an $L^p$-operator algebra, see Definition \ref{Lp}, concretely represented on $L^p(\mu_{\geq 0})$ by its canonical inclusion map. The same is true for $A(d,\rho)$, so in the sequel we always assume that these algebras act on $L^p(\mu_{\geq 0})$ with no need 
of explicitly giving a representation. 
\end{remark}

Our main goal is to define a Dirac operator on certain algebras of the form $A(d, \rho)$. To do so, we first prove a structural result for the measure spaces involved, Proposition \ref{iotas_pis}, which in particular highlights the structure of subspaces and corresponding projections.

We start by defining, for each $n \in \Z_{\geq 0}$, 
an inclusion map  $\iota_n \colon L^p(\mu_{\leq n}) \hookrightarrow L^p(\mu_{\geq 0})$ and a 
projection map $\pi_n \colon L^p(\mu_{\geq 0}) \twoheadrightarrow L^p(\mu_{\leq n})$. 

\begin{definition}\label{iotaspisPs}
Fix $p \in [1, \infty)$, let $(X_j, \mathfrak{M}_j, \mu_j)_{j=0}^\infty$
be a sequence of probability spaces, and let $(X_{\leq n}, \mathfrak{M}_{\leq n}, \mu_{\leq n}) $ and 
$(X_{\geq n}, \mathfrak{M}_{\geq n}, \mu_{\geq n})$ be as in Definition \ref{Prob_leq_>}.  For each $n \in \Z_{\geq 0}$ we define 
\begin{enumerate}
\item \label{Incl_n} $\iota_n \colon L^p(\mu_{\leq n}) \to L^p(\mu_{\geq 0})$ so that if $\xi \in L^p(\mu_{\leq n})$, $(x_0, \ldots, x_n) \in X_{\leq n}$, and $\vf{x} \in X_{\geq n+1}$, then 
\[
(\iota_n\xi)(x_0, \ldots, x_n, \vf{x}) \coloneqq  \xi(x_0, \ldots, x_n).
\]
\item \label{Proj_n} $\pi_n \colon  L^p(\mu_{\geq 0}) \to L^p(\mu_{\leq n})$
by putting, for every $\eta \in L^p(\mu_{\geq 0})$ and $(x_0, \ldots, x_n) \in X_{\leq n}$, 
\[
(\pi_n\eta)(x_0, \ldots, x_n) \coloneqq  \int_{X_{\geq n+1}} \eta(x_0, \ldots, x_n, \vf{x}) d\mu_{\geq n+1}(\vf{x}) 
\]
\item \label{Pprojns} $P_n \colon L^p(\mu_{\geq 0}) \to  L^p(\mu_{\geq 0})$ by $P_n \coloneqq  \iota_n \circ \pi_n$. 
\end{enumerate}
\end{definition}
We now describe the properties of the maps $\iota_n, \pi_n,$ and $P_n$. 
\begin{proposition}\label{iotas_pis}
Fix $p \in [1, \infty)$, let $(X_j, \mathfrak{M}_j, \mu_j)_{j=1}^\infty$
be a sequence of probability spaces, and let $(X_{\leq n}, \mathfrak{M}_{\leq n}, \mu_{\leq n}) $ and 
$(X_{> n}, \mathfrak{M}_{> n}, \mu_{> n})$ be as in Definition \ref{Prob_leq_>}.  For each $n \in \Z_{\geq 0}$, consider the maps $\iota_n, \pi_n,$ and $P_n$
defined in Definition \ref{iotaspisPs} above. 
Then, 
\begin{enumerate}
\item $\iota_n \in \mathcal{B}(L^p(\mu_{\leq n}) , L^p(\mu_{\geq 0}))$ and $\iota_n$ is an isometry. \label{iotaISOM}
\item $\pi_n \in \mathcal{B}(L^p(\mu_{\geq 0}), L^p( \mu_{\leq n}) )$ and $\|\pi_n\| \leq 1$.
\label{Picontr} 
\item $P_n \in \mathcal{B}(L^p(\mu_{\geq 0}))$ and if $m \in \Z_{\geq n}$, then  $P_nP_m=P_mP_n=P_n$. In particular this gives  $P_n^2=P_n$ and $\|P_n\| = 1$. 
\label{Pprojnsbdd}
\item $\pi_n \circ \iota_n = \op{Id}_{\leq n} \in \Li(L^p(\mu_{\leq n}))$. In particular, $\iota_n$ is injective and $\pi_n$ is surjective. 
\label{pi_iot=1}
\end{enumerate}
\end{proposition}
\begin{proof}
It is clear that $\xi \mapsto \iota_n \xi$ is linear. 
Since we are working with probability spaces, an application of Fubini's theorem gives
\[
\int_{X_{ \geq 0}} |(\iota_n\xi)(x_0, \ldots, x_n, \vf{x})|^p d\mu_{\geq 0}(x_0, \ldots, x_n, \vf{x})
= 
\int_{X_{\leq n}} |\xi|^p d\mu_{\leq n},
\]
proving that $\iota_n \colon L^p(\mu_{\leq n}) \to  L^p(\mu_{\geq 0})$ is indeed isometric, 
whence part \eqref{iotaISOM} follows. 
Next, it is also clear that $\eta \mapsto \pi_n \eta$ is a linear map. Moreover, 
Jensen's inequality and Fubini's theorem give
\begin{align*}
\| \pi_n \eta \|_p^p & = \int_{X_{\leq n}} \left| \int_{X_{\geq n+1}} \eta(x_0, \ldots, x_n, \vf{x}) d\mu_{\geq n+1}(\vf{x}) \right|^pd\mu_{\leq_n}(x_0, \ldots, x_n)\\
& \leq \int_{X_{\leq n}}\left( \int_{X_{\geq n+1}} |\eta(x_0, \ldots, x_n, \vf{x})|^p d\mu_{\geq n+1}(\vf{x}) \right)d\mu_{\leq n}(x_0, \ldots, x_n) \\
& = \int_{X_{\geq 0}} |\eta|^p d\mu_{\geq 0}\\
& = \| \eta \|_p^p.
\end{align*}
Therefore $\| \pi_n\| \leq 1$, proving part \eqref{Picontr}. 
For part \eqref{Pprojnsbdd}, Fubini's theorem together with $n\leq m$, 
shows that, for any $\eta \in L^p(\mu_{\geq 0})$, 
$(x_0, \ldots, x_n) \in X_{\leq n}$, and $\vf{x} \in X_{\geq n+1}$, 
\begin{align*}
(P_nP_m\eta)(x_0, \ldots, x_n, \vf{x}) & = \int_{X_{(n,m]}} (\pi_m \eta)(x_0, \ldots, x_n, \vf{y}) d\mu_{(n,m]}(\vf{y}) \\
& = \int_{X_{\geq m+1}}\left( \int_{X_{(n,m]}} \eta(x_0, \ldots, x_n, \vf{y}, \vf{z}) d\mu_{(n,m]}(\vf{y}) \right) d\mu_{\geq m+1}(\vf{z})\\
 & = \int_{X_{\geq n+1}} \eta(x_0, \ldots, x_n, \vf{w}) d\mu_{\geq n+1}(\vf{w}) \\
& = 
(\pi_n\eta)(x_0, \ldots, x_n) \\
& = (P_n\eta)(x_0, \ldots, x_n, \vf{x}).
\end{align*}
This gives $P_nP_m=P_n$. A straightforward computation shows that $m \geq n$ also implies $P_mP_n=P_n$, 
so part \eqref{Pprojnsbdd} is proved. Finally, to show part \eqref{pi_iot=1}, take any 
$\xi \in L^p(\mu_{\leq n})$ and any $(x_0, \ldots, x_n) \in X_{\leq n}$, and compute 
\[
((\pi_n \circ \iota_n)\xi)(x_0, \ldots, x_n) 
= 
\int_{X_{\geq n+1}} (\iota_n\xi)(x_0, \ldots, x_n, \vf{x}) d\mu_{\geq n+1}(\vf{x}) 
= \xi(x_0, \ldots, x_n).
\]
 That is, $(\pi_n \circ \iota_n)\xi=\xi$, finishing the proof. 
\end{proof}

The following lemma will be needed to produce $L^p$-analogues
of spectral triples on UHF $L^p$-operator algebras. 

\begin{lemma}\label{aP=Pa}
Let $A \coloneqq A(d,\rho) \subseteq \Li(L^p(\mu_{\geq 0}))$ be an UHF $L^p$-operator algebra as defined in Equation \eqref{UHFLP} 
and, for each $n \in \Z_{\geq 0}$, 
let $P_n \in \Li(L^p(\mu_{\geq 0}))$ be as in part \eqref{Pprojns} 
of Definition \ref{iotaspisPs}. If $m, n \in \Z_{\geq 0}$ are such that $m\geq n$ and $a \in A_n  \subseteq \Li(L^p(\mu_{\geq 0}))$, then $aP_m=P_ma$. 
\end{lemma}
\begin{proof}
Fix $m,n \in \Z_{\geq 1}$ with $n \leq m$. In analogy with the map $\iota_n$ defined in Definition \ref{iotaspisPs} we also get a map $\iota_{n,m} \colon L^p(\mu_{(n,m]}) \to L^p(\mu_{\geq n+1})$ (see Definition \ref{Prob_leq_>}) defined as follows: 
for $\xi \in L^p(\mu_{(n,m]})$, $(x_{n+1}, \ldots, x_m) \in X_{(n,m]}$, 
and $\vf{x} \in X_{\geq m+1}$ we put
\[
(\iota_{n,m}\xi)(x_{n+1}, \ldots, x_m, \vf{x}) \coloneqq   \xi(x_{n+1}, \ldots, x_m).
\]
If $n=m$, we set $(\iota_{m,m}\xi)(\vf{x})=1$ for all  $\vf{x} \in X_{\geq m+1}$.
Similarly, as with $\pi_n$, we now define $\pi_{m,n} \colon L^p(\mu_{\geq n+1}) \to L^p(\mu_{(n,m]})$ so that for $\zeta \in  L^p(\mu_{\geq n+1})$ and $(x_{n+1}, \ldots, x_m) \in X_{(n,m]}$
we set
\[
(\pi_{m,n}\zeta)(x_{n+1}, \ldots, x_m) \coloneqq  \int_{X_{\geq m+1}} \zeta(x_{n+1}, \ldots, x_m, \vf{x}) d\mu_{\geq m+1}(\vf{x}).
\]
Note that when $n=m$ this is interpreted as 
\[
\pi_{m,m}\zeta =  \int_{X_{\geq m+1}} \zeta(\vf{x}) d\mu_{\geq m+1}(\vf{x}) \in \C = L^p(\mu_{(m,m]}).
\]
Thus, we can define $P_{m,n} \coloneqq   \iota_{n,m} \circ \pi_{m,n} \in 
\Li(L^p(\mu_{\geq n+1}))$. Furthermore, if $\eta_1 \otimes \eta_2 \in L^p(\mu_{\leq n}) \otimes_p L^p(\mu_{\geq n+1})=L^p(\mu_{\geq 0})$, then a direct computation gives 
\begin{equation}\label{PpushTensor}
P_m(\eta_1 \otimes \eta_2) = \eta_1 \otimes P_{m,n}(\eta_2).
\end{equation}
We are now ready to show that $aP_m=P_ma$ when $a \in A_n$ and $n \leq m$.
Since the span of elementary tensors in $L^p(\mu_{\leq n}) \otimes_p L^p(\mu_{\geq n+1})$ is dense in $L^p(\mu_{\geq 0})$ (see Theorem \ref{SpatialTP}(\ref{denseTensor_p}) and the identification in Equation \eqref{tensorDecomp}), 
we only need to prove that 
\[
aP_m(\eta_1 \otimes \eta_2)=P_ma(\eta_1 \otimes \eta_2),
\]
when $\eta_1 \in L^p(\mu_{\leq n})$ and 
$\eta_2 \in L^p(\mu_{\geq n+1})$. Indeed, 
since $a \in A_n$, there is $b \in M_{n}$ such that $a=\rho_{\infty,n}(b)=\rho_{\leq n}(b) \otimes \op{Id}_{\geq n+1}$. Then, since $n \leq m$, two applications of 
Equation \eqref{PpushTensor}
give
\begin{align*}
aP_m(\eta_1 \otimes \eta_2) &= a(\eta_1 \otimes P_{m, n}(\eta_2)) \\
&= \rho_{\leq n}(b)\eta_1 \otimes P_{m, n}(\eta_2) \\
&= P_m(\rho_{\leq n}(b)\eta_1 \otimes \eta_2) \\
&= P_ma(\eta_1 \otimes \eta_2),
\end{align*}
finishing the proof. 
\end{proof}

For $p \in [1, \infty)$ we now proceed to define an unbounded operator $D$ on a dense suspace of $L^p(\mu_{\geq 0})$ such that 
 $(L^p(\mu_{\geq 0}), D)$ is an $L^p$-analogue of an unbounded Fredholm module over $A(d,\rho) \subseteq \Li(L^p(\mu_{\geq 0}))$. 
 We do so by mimicking part of the construction in Theorem 2.1 
 in \cite{CI} for general AF C*-algebras.  At this point, we need to add 
 extra structure on the representations $\rho=(\rho_j)_{j=0}^\infty$
 used to construct $L^p$ UHF-algebras of tensor type (see Equation \eqref{UHFLP}). The following assumptions will be used in the sequel:
\begin{enumerate}[label=(\subscript{U}{{\arabic*}})]
 \item \label{A1} The measure space $(X_0, \mathfrak{M}_0, \mu_0)$ is one dimensional, that is $L^p(\mu_0)$ is $\ell_1^p = \C$ and $d(0)=1$, 
 (we still ask $d(n) \geq 2$ for all $n \geq 1$). As a consequence, this forces the representation $\rho_0 \colon \C \to \Li(\ell_1^p)$ to be the trivial isometric one given by $\rho_0(z)\zeta = z\zeta$ for any $z, \zeta \in \C$.
 \item \label{A2} The given sequence of probability spaces $(X_j, \mathfrak{M}_j, \mu_j)_{j=0}^\infty$ consists of compact metric probability spaces. 
 That is, we assume that each $X_j$ is a compact metric space, with metric henceforth 
 denoted by $\dd_j \colon X_j \times X_j \to \R_{\geq 0}$, equipped with the Borel sigma 
 algebra and a probability measure $\mu_j$ on it. Note that thanks to \ref{A1}, $X_0=\{1\}$ already satisfies this assumption where $\dd_0=0$.
 \item \label{A3} Each $\dd_j$ is bounded by $1$, so that the infinite product measure spaces from Definition \ref{Prob_leq_>} \eqref{InfProd} are also compact metric probability space whose associated metric $\dd_{\geq n} \colon X_{\geq n} \times X_{\geq n} \to \R_{>0}$ is given by 
 \begin{equation}\label{ddMet}
  \dd_{\geq n}(\vf{x}, \vf{y}) = \dd_{ \geq n}((x_{n}, x_{n+1}, \ldots), (y_{n}, y_{n+1}, \ldots)) \coloneqq  \sum_{k=n}^\infty \frac{1}{2^k}\dd_k(x_{k}, y_k)
 \end{equation}
 \end{enumerate}
In Theorem \ref{UHFpST} below we will also need a fourth assumption:
 \begin{enumerate}[label=(\subscript{U}{{\arabic*}})] \setcounter{enumi}{3}
 \item \label{A4} The spaces $L^p(\mu_j)$ are all finite dimensional. That is, 
 we will assume that each $X_j$ consists 
 of finitely many points with $\mu_j$ being normalized counting measure. 
\end{enumerate}
\begin{lemma}\label{P_n-id_strongly}
Let $(X_j, \mathfrak{M}_j, \mu_j)_{j=0}^\infty$ be a sequence 
of compact-metric probability spaces satisfying assumptions \ref{A1}-\ref{A3} above. 
Let $p \in [1, \infty)$ and let $P_n=\iota_n \circ \pi_n \in \Li(L^p(\mu_{\geq 0}) )$ be as defined in part \eqref{Pprojns} 
of Definition \ref{iotaspisPs}. Then $P_n$ converges strongly to $\op{Id}_{\geq 0}$.
\end{lemma}
\begin{proof}
Note that $X_{\geq 0}$ is compact and that $C(X_{\geq 0})$ is a dense subset of $L^p(\mu_{\geq 0})$. Thus,  
it suffices to show that $\| P_n \eta - \eta \|_p \to 0$ as $n \to \infty$ 
for any $\eta \in C(X_{ \geq 0})$. Let $\varepsilon >0$, then by uniform continuity there is $\delta >0$
such that for all $\vf{x}, \vf{y} \in X_{ \geq 0}$, $\dd_{ \geq 0}(\vf{x}, \vf{y}) < \delta$ 
implies that $|\eta(\vf{x})-\eta(\vf{y})|<\varepsilon$.  
Take a finite cover $(B_{\delta/3}(\vf{x}_j))_{j=1}^k$ of $X_{ \geq 0}$. 
For every $n \in \Z_{\geq 0}$, let $p_n \colon X_{ \geq 0} \to X_{\geq n+1}$ be defined by
\[
p_n(\vf{x}) \coloneqq  p_n(x_0, x_1, \ldots, x_n, x_{n+1}, \ldots)=(x_{n+1}, x_{n+2}, \ldots).
\]
 For 
every $j,l \in \{1, \ldots, k\}$, we use convergence of the series in \eqref{ddMet} to find an integer $N_{j,l}>0$ such that whenever 
$n\geq N_{j,l}$
\[
\dd_{\geq n+1}(p_n(\vf{x}_j), p_n(\vf{x}_l)) < \frac{\delta}{3}.
\]
Set $N= \max\{N_{j,l} \colon 1 \leq j,l \leq  k\}$, let $n \geq N$, 
and take $\vf{x}, \vf{y} \in X_{\geq 0}$ such that $x_0=y_0$, $x_1=y_1$, $\ldots$, $x_n=y_n$. This implies that $\dd_{\geq 0}(\vf{x}, \vf{y})=\dd_{\geq n+1}(p_n(\vf{x}), p_n(\vf{y}))$. Further, there are $j,l \in \{1, \ldots, k\}$
such that $\vf{x} \in B_{\delta/3}(\vf{x}_j)$ and $\vf{y} \in B_{\delta/3}(\vf{y}_l)$. 
Therefore,
\begin{align*}
\dd_{\geq 0}(\vf{x}, \vf{y}) & = \dd_{\geq n+1}(p_n(\vf{x}), p_n(\vf{y})) \\
& \leq  \dd_{\geq n+1}(p_n(\vf{x}), p_n(\vf{x}_j)) + \dd_{\geq n+1}(p_n(\vf{x}_j), p_n(\vf{x}_l)) + \dd_{\geq n+1}(p_n(\vf{x}_l), p_n(\vf{y})) \\
& \leq \dd_{\geq 0}(\vf{x}, \vf{x}_j) + \dd_{\geq n+1}(p_n(\vf{x}_j), p_n(\vf{x}_l)) + \dd_{\geq 0}(\vf{x}_l, \vf{y}) \\
& < \delta.
\end{align*}
Hence, for each $n \geq N$, if $\vf{z}, \vf{w} \in X_{\geq n+1}$, we have 
\[
|\eta(x_0, \ldots, x_n, \vf{z})-\eta(x_0, \ldots, x_n, \vf{w})|<\varepsilon.
\]
This in turn implies that for any $\vf{x}$, if $n \geq N$, 
\[
|(P_n\eta)(\vf{x})-\eta(\vf{x})| \leq \int_{X_{\geq n+1}}
|\eta(x_1, \ldots, x_n, \vf{z})-\eta(x_1, \ldots, x_n, p_n(\vf{x}))| d\mu_{\geq n+1}(\vf{z}) < \varepsilon,
\]
from where it follows that $\| P_n\eta - \eta\|_p \leq \varepsilon$ for each $n \geq N$. 
This finishes the proof. 
\end{proof}
An immediate consequence of the previous lemma is that any element $\eta \in L^p(\mu_{\geq 0})$ is approximated by the sequence $(P_n\eta)_{n=0}^\infty$ in $L^p(\mu_{\geq 0})$. For each $n \in \Z_{\geq 0}$ we have $P_n\eta\in \iota_n(L^p(\mu_{\leq n}))$, whence it follows that the nested union 
\begin{equation}\label{NestedD}
\mathcal{D} \coloneqq  \bigcup_{n=0}^\infty \iota_n(L^p(\mu_{\leq n}))
\end{equation}
is a dense subspace of $L^p(\mu_{\geq 0})$.
\begin{definition}\label{Qmaps}
Fix $p \in [1, \infty)$, let $(X_j, \mathfrak{M}_j, \mu_j)_{j=0}^\infty$
be a sequence of probability spaces, and let $(X_{\leq n}, \mathfrak{M}_{\leq n}, \mu_{\leq n}) $ and 
$(X_{\geq n}, \mathfrak{M}_{\geq n}, \mu_{\geq n})$ be as in Definition \ref{Prob_leq_>}.  For each $n \in \Z_{\geq 1}$ we define 
\[
Q_n \coloneqq  P_n-P_{n-1} \in \Li(L^p(\mu_{\geq 0})).
\]
Additionally we define $Q_0 \coloneqq  P_0$. 
\end{definition}
\begin{lemma}\label{QinK}
For every $n \in \Z_{\geq 0}$, $P_n, Q_n \in \mathcal{K}(L^p(\mu_{\geq 0}))$, 
provided that we assume condition \ref{A4} above. 
 \end{lemma}
 \begin{proof}
Part \eqref{Pprojnsbdd} in Proposition \ref{iotas_pis} gives at once that, 
for any $n \in \Z_{\geq 0}$, 
\[
P_n(L^p(\mu_{\geq 0})) = \iota_n(L^p(\mu_{\leq n})) \cong \bigotimes_{j=0}^n L^p(\mu_j).
\]
Hence, assumption \ref{A4} implies that $P_n$ has finite rank, and therefore 
for each $n \in \Z_{\geq 0}$, we must have $Q_n= P_n-P_{n-1}\in \mathcal{K}(L^p(\mu_{\geq 0}))$, as wanted. 
 \end{proof}
 Fix a sequence $\alpha=(\alpha_n)_{n=1}^\infty$
in $\R_{>0}$ with $\alpha_n \to \infty$ and set $\alpha_0=0$. 
Let $\mathcal{D}$ be as in Equation \eqref{NestedD} and define $D_\alpha  \colon \mathcal{D} \to L^p(\mu_{\geq 0})$ by 
\begin{equation}\label{UHFdirac}
D_\alpha \coloneqq  \sum_{n=1}^\infty \alpha_n Q_n.
\end{equation}
Observe first that $D_\alpha$ is well defined. Indeed, if $\eta \in \mathcal{D}$, there is $k \in \Z_{\geq 1}$ such that $\eta \in \iota_k(L^p(\mu_{\leq k}))$. Next, notice that $P_m\eta=P_{m-1}\eta=\eta$ for all $m > k$. Therefore 
\[
D_\alpha\eta = \sum_{n=1}^k \alpha_n Q_n\eta \in L^p(\mu_{\geq 0}).
\]
We have now arrived at an analogue of Theorem \ref{GdisLpT} for $L^p$ UHF-algebras 
of tensor type.
\begin{theorem}\label{UHFpST}
Let $p \in [1, \infty)$, let $A \coloneqq A(d,\rho) \subseteq \Li(L^p(\mu_{\geq 0}))$ be an $L^p$ UHF-algebra as defined in 
\eqref{UHFLP} so that $\rho=(\rho_j)_{j=0}^\infty$ satisfies the assumptions \ref{A1}-\ref{A4} above. Let $D=D_\alpha$ be as in \eqref{UHFdirac}. Then, $(A, L^p(\mu_{\geq 0}), D)$ is an $L^p$-spectral triple 
as in Definition \ref{LpST}. That is, 
\begin{enumerate}
\item \label{UHFpST1} $(\op{Id}_{ \geq 0}+D^2)^{-1} \in \mathcal{K}(L^p(\mu_{\geq 0}))$, 
\item \label{UHFpST2} $(D-\lambda \op{Id}_{ \geq 0})^{-1} \in \mathcal{K}(L^p(\mu_{\geq 0}))$ for all $\lambda \in \C \setminus \sigma(D)$. 
\item \label{UHFpST3}$\{ a \in A \colon \| [D,a] \| < \infty\}$ is 
a norm dense subalgebra of $A$.
\end{enumerate}
\end{theorem}
\begin{proof}
Notice that for any $m>n$, we use part \eqref{Pprojnsbdd} in 
Proposition \ref{iotas_pis} to compute 
\[
Q_nQ_m
= P_nP_m-P_nP_{m-1}-P_{n-1}P_m+P_{n-1}P_{m-1}
= P_n-P_n-P_{n-1}+P_{n-1}=0, 
\]
and similarly $Q_mQ_n=0$. Furthermore, for each $n \in \Z_{\geq 0}$, 
\[
Q_n^2=P_n^2-P_nP_{n-1}-P_{n-1}P_n+P_{n-1}^2=P_n-2P_{n-1}+P_{n-1}=Q_n.
\]
Furthermore observe that on the domain $\mathcal{D}$ (see Equation \eqref{NestedD}) we get 
\[
\op{Id}_{ \geq 0}= P_0 + \sum_{n=1}^\infty Q_n.
\]
Hence, on $\mathcal{D}$, 
\begin{equation}\label{I+D^2onD}
\op{Id}_{ \geq 0}+D^2  =\op{Id}_{ \geq 0} + \sum_{n=1}^\infty \alpha_n^2 Q_n= P_0 + \sum_{n=1}^\infty (1+\alpha_n^2) Q_n.
\end{equation}
We now define, for each $m \in \Z_{\geq 0}$, a
map $K_m \in \Li(L^p(\mu_{\geq 0}))$ by
\[
K_m \coloneqq  \sum_{n=1}^m \frac{1}{1+\alpha_n^2} Q_n.
\]
Notice that, since $\alpha_n\to \infty$ and since for each $n \in \Z_{\geq 0}$, $Q_n$ has finite rank (see Lemma \ref{QinK}), then $(K_m)_{m=1}^\infty$ is a Cauchy sequence in $\mathcal{K}(L^p(\mu_{>0}))$, 
so it converges to an element $K \in \mathcal{K}(L^p(\mu_{>0}))$. 
A direct computation, using Equation \eqref{I+D^2onD} together with the fact that $P_0Q_n=0$ for all $n \in \Z_{\geq 0}$ (see Part \eqref{Picontr} in Proposition \ref{iotas_pis}), gives  
\[
\op{Id}_{\mathcal{D}}=K(\op{Id}_{ \geq 0}+D^2) = (\op{Id}_{ \geq 0}+D^2)K,
\]
which proves that $\op{Id}_{ \geq 0}+D^2$ is invertible 
on its range, which is in turn contained in $\mathcal{D}$ as per Equation \eqref{I+D^2onD}. Therefore, using the density of $\mathcal{D}$ in $L^p(\mu_{\geq 0})$, we get that $(\op{Id}_{ \geq 0}+D^2)^{-1} = K \in \mathcal{K}(L^p(\mu_{>0}))$, proving Part \eqref{UHFpST1}. Part \eqref{UHFpST2} 
is proved analogously by observing that $\sigma(D)=\{\alpha_n \colon n \in \Z_{\geq 0}\}$.  
Finally, for Part \eqref{UHFpST3} since by definition of $A$ (see Equation \eqref{UHFLP}) the union of the algebras $A_n$ is dense in $A$, 
it suffices to show that $[D,a]$ is bounded for any $a \in A_n$. 
To do so, we fix $n \in \Z_{\geq 0}$ and take any $a \in A_n$. It follows from Lemma
\ref{aP=Pa} that, for any $m>n$, we have $aP_m=P_ma$ and that $aP_{m-1}=P_{m-1}a$. 
Then, 
\[
[Q_m, a] = Q_ma-aQ_m = (P_m-P_{m-1})a - a(P_m-P_{m-1})=0.
\]
Therefore, 
\[
[D,a] = \sum_{k=0}^n \alpha_k [Q_k, a] 
\]
and since each $ [Q_k, a] \in \Li(L^p(\mu_{\geq 0}))$ it follows that $\|[D,a]\| < \infty$, as wanted. 
\end{proof}

\subsection{$p$-Quantum Compact Metric for UHF $L^p$-Spectral Triples.}\label{QMUHF}

Theorem 2.1 in \cite{CI} is an important result, where it is shown that the Dirac operator construction on a general AF-algebra C*-algebra $A$ gives rise to a quantum compact metric on $S(A)$, the state space of $A$, which induces the weak-$*$ topology on it. Below we prove that a 
$p$-analogue of this result is also true. We do mention that, due to a lack of a GNS representation for $L^p$-operator algebras, we cannot carry Christensen-Ivan's construction word by word. However, thanks to the work presented in Subsection \ref{DOUHF_ST} above, a slight modification of the argument used in the proof of \cite[Theorem 2.1]{CI} works in the $L^p$-setting. 
Our technique will be to use  M.\ A.\ Rieffel's general scenario described at the end of Section \ref{secDefiLp} (Theorem \ref{RSuffT}) applied to the $L^p$ setting, which will guarantee that the  quantum compact metric construction can also be carried out. 

We are now ready to state and prove our version of Theorem 2.1 in \cite{CI}: 

\begin{theorem}\label{FinThm}
Let $p \in [1, \infty)$ and let $(A(d, \rho), L^p(\mu_{\geq 0}), D_\alpha)$ be the associated $L^p$-spectral triple from Theorem \ref{UHFpST}.  Then, there exists a sequence of real numbers $\alpha=(\alpha_n)_{n=0}^\infty$ such that $\alpha_n \to \infty$ and such that the extended pseudometric $\op{mk}_{D_\alpha}$ (defined by Equation \eqref{ExtPM}) is an actual metric that induces the weak-$*$ topology on $S(A)$.
\end{theorem}
\begin{proof}
We will show that all the hypotheses of  M.\ A.\ Rieffel's theorem (Theorem \ref{RSuffT})
are met. 
To do so, we start by checking that all the assumptions \ref{R1}-\ref{R5} are satiesfied. Indeed, the normed space $B$ is now $A(d,\rho) \subseteq \Li(L^p(\mu_{\geq 0}))$, any $L^p$ UHF-algebra defined in \eqref{UHFLP} that also satisfies assumptions \ref{A1}-\ref{A3} above. 
For now, we fix a sequence $\alpha=(\alpha_{n})_{n=1}^\infty$ in $\R_{>0}$ 
and set $\alpha_0=0$. Recall that $D=D_\alpha$ is as in Equation \eqref{UHFdirac}
above. 
Set
\[
\mathcal{L} \coloneqq  \{ a \in A \colon \| [D, a] \| < \infty \}
\]
as our subspace and let $L\colon \mathcal{L} \to \R_{\geq 0}$ be given by 
$L(a) =  \| [D, a] \|$.
Then, by construction of $D$, we have 
\[
\mathcal{N} = \{a \in  \mathcal{L} \colon L(a) = 0\} = \C \cdot 1_{A_0} = \C \cdot 1_{A}  \cong \C.
\]
Next we define $\varphi \in A'$ by letting 
$\varphi(a)=1$ for all $a \in \C \cdot 1_{A}$ and zero elsewhere. 
Thus, 
\[
 S \coloneqq \{\omega \in  A'  \colon \omega = \varphi \text{ on $\mathcal{N}$, and }\|\omega\|= 1\}
=  \{\omega \in  A'  \colon \|\omega\|= \omega(1_A)=1\}.
\] 
This gives that $S=S(A)$ is the standard definition of the state space of a unital Banach algebra $A$ as in Definition \ref{States}. In order for assumption \ref{R5} to be satisfied, 
we now need to verify that $\mathcal{L}$ does indeed separate the points of $S(A)$. 
Indeed, take any $\omega, \psi \in S(A)$ with $\omega \neq \psi$, but suppose 
for a contradiction that $\omega(b)=\psi(b)$ for all $b \in \mathcal{L}$.  Now since Theorem \ref{UHFpST} guarantees that $\mathcal{L}$ is dense in $A$, so we must have that $\omega$ and $\psi$ actually agree on all of $A$, 
contradicting that $\omega \neq \psi$. 

Note that $\op{d}_{S(A)}$ and $\op{mk}_D$ are now the same function. Thus, it only remains for us to prove that the image of $\mathcal{L}_1 = \{ a \in \mathcal{L} \colon L(a)\leq 1\}$ in $\mathcal{L}/\C$ is totally bounded for the quotient norm $\|-\|_{A/\C}$. To do so, 
it suffices to show that a key estimate from the proof of \cite[Theorem 2.1]{CI}
remains valid in this situation. Since we don't have a GNS representation here, 
we also don't have the separating vector used in \cite{CI}. 
However, since $(X_{\geq0}, \mu_{\geq0})$ 
is a probability space, we do have  that the constant function $1$ is  in $L^p(\mu_{\geq 0})$. 
Hence, for any $a \in \mathcal{L}_1$ we get 
\[
1 \geq \| [D,a] \| \geq \| Q_{n+1}[D, a] Q_01\|_p = |\alpha_{n+1}|\|Q_{n+1}a1\|_p.
\] 
By dimension arguments (see Lemma \ref{QinK}) we now have that
the seminorms $a \mapsto \|Q_{n+1}a1\|_p$ and $ a \mapsto \|Q_{n+1}a\|$
are equivalent so we are able to find $c_{n+1} \in R_{\geq 1}$
such that 
\[
\| Q_{n+1}(a)\| \leq \frac{c_{n+1}}{|\alpha_{n+1}|}.
\]
The choice for $(\alpha_n)_{n=0}^\infty$ is now made exactly as in \cite{CI}, 
finishing the proof. 
\end{proof}

Once again, in terms of Definition \ref{pQMetric} we have shown the following result:

\begin{corollary}
Let $p \in [1, \infty)$. Then  there exists a sequence of real numbers $\alpha=(\alpha_n)_{n=0}^\infty$ with $\alpha_n \to \infty$ and such that $(A(d, \rho), L^p(\mu_{\geq 0}), D_\alpha)$, the associated $L^p$-spectral triple from Theorem \ref{UHFpST},  is metric. 
\end{corollary}

\section{A unifying example}\label{Gp_UHF}

In this section we will present an example of a bounded-doubling length function on a group, whose $L^p$-group algebra is naturally a UHF $L^p$-operator algebra. 
This example combines and unifies instances of metric UHF spectral triples (in the sense of E.\ Christensen and C.\ Ivan \cite{CI}) with metric spectral triples associated to length functions with bounded doubling on groups (as studied by M. \ Christ and  M.\ A.\ Rieffel \cite{ChristRieffel17}). Further, we mention that this example is already of interest in the classical C*-case. 

We begin with the following, not immediately obvious 
strict inequality, that will be needed 
at some point in the proof of 
Proposition \ref{L_nBD}.
\begin{lemma}\label{SinLemma}
If $x \geq \frac{4}{3}$, then $2\sin(\frac{\pi}{x}) > \frac{1}{x}$. 
\end{lemma}
\begin{proof}
We will show instead that $2\sin(\pi y) > y$ 
for any $y \in (0, \frac{3}{4}]$. The desired result 
will then follow by taking $y = \frac{1}{x}$. 

If $y=0$, then $2\sin(\pi y) = y$, 
and if $y=\frac{3}{4}$, then 
$2\sin(\pi y) > y$. Now let 
$f \colon [0, \frac{3}{4}) \to \R$ be given by 
$f(y)\coloneqq 2\sin(\pi y) - y$. Since $f''(y)=-2\pi^2\sin(\pi y)$, 
it follows that $f''(y)<0$ for all $y \in (0, \frac{3}{4}]$
and therefore the graph of $f$ is concave downward. 
Thus, the inequality $2\sin(\pi y) > y$ is preserved
on all of $ (0, \frac{3}{4}]$, as we wanted to prove. 
\end{proof}

For a fixed $n \in \Z_{>1}$ we first consider the group $G_n\coloneqq\bigoplus_{j =1}^\infty \Z/n\Z$. 
Our first goal is to define a length function on $G_n$. 
To do so, let $\mathbb{L}_n \colon G_n \to \R_{\geq 0}$ be given by 
\begin{equation}\label{LonZ}
\mathbb{L}_n\left((m_j)_{j=1}^\infty\right) \coloneqq 
\sum_{j=1}^\infty \left| \exp\left(\frac{2\pi im_j}{n}\right)-1 \right|n^j.
\end{equation}
We remark that, by definition of direct sums of groups, the expression 
in the right hand side of (\ref{LonZ}) is always a finite sum. 
\begin{proposition}\label{L_nBD}
The function $\mathbb{L}_n \colon G_n \to \R_{\geq 0}$ given by (\ref{LonZ})
is a proper length function on $G_n=\bigoplus_{j =1}^\infty \Z/n\Z$ of bounded $n$-dilation
with $K_{\LL_n} = n^4$ (see Definition \ref{LengthF}).
\end{proposition}
\begin{proof}
First we show that $\mathbb{L}_n$ is indeed a length function on $G_n$. 
It's clear that $\mathbb{L}_n((0)_{j=1}^\infty)=0$. Now 
assume that $\mathbb{L}_n((m_j)_{j=1}^\infty) =0$, then 
$| \exp(\frac{2\pi im_j}{n})-1|n^j=0$ for all $j \in \Z_{\geq 1}$, 
and therefore $m_j \equiv n \pmod{n}$ for every $j  \in \Z_{\geq 1}$, which 
gives that $(m_j)_{j=1}^\infty$ is the identity of $G_n$. Next, for 
any  $(m_j)_{j=1}^\infty \in G_n$, 
\begin{align*}
\mathbb{L}_n\left((-m_j)_{j=1}^\infty\right) & = 
\sum_{j=1}^\infty \left| \exp\left(\frac{-2\pi im_j}{n}\right)-1 \right|n^j \\
&= \sum_{j=1}^\infty \left| \exp\left(\frac{2\pi im_j}{n}\right)-1 \right|n^j \\
& = \mathbb{L}_n\left((m_j)_{j=1}^\infty\right).
\end{align*}
Next, if $(m_j)_{j=1}^\infty,(m'_j)_{j=1}^\infty  \in G_n$, observe that for each $j \in \Z_{\geq 1}$,
\[
\exp\left(\frac{2\pi i(m_j+m'_j)}{n}\right)-1 = \exp\left(\frac{2\pi i m_j}{n}\right)\left( \exp\left(\frac{2\pi i m'_j}{n}\right) - \exp\left(\frac{-2\pi i m_j}{n}\right)\right),
\]
whence 
\begin{align*}
\left| \exp\left(\frac{2\pi i(m_j+m'_j)}{n}\right)-1 \right| & = \left| \exp\left(\frac{2\pi i m'_j}{n}\right) - \exp\left(\frac{-2\pi i m_j}{n}\right)\right| \\ 
& = \left| \exp\left(\frac{2\pi i m'_j}{n}\right) -1 - \exp\left(\frac{-2\pi i m_j}{n}\right)+1\right|\\
& \leq  \left| \exp\left(\frac{2\pi im_j}{n}\right)-1 \right| + \left| \exp\left(\frac{2\pi im'_j}{n}\right)-1 \right|.
\end{align*}
This gives at once that 
\[
\mathbb{L}_n\left((m_j+m'_j)_{j=1}^\infty\right) \leq \mathbb{L}_n\left((m_j)_{j=1}^\infty\right)  + \mathbb{L}_n\left((m'_j)_{j=1}^\infty\right),
\]
implying that $\mathbb{L}_n$ is indeed a length function on $G$. 

In order to show that $\mathbb{L}_n$ is proper, it suffices to show that 
for each $k \in \Z_{\geq 1}$,
\begin{equation}\label{B(n^k)}
\left\{(m_j)_{j=1}^\infty \in G_n \colon \mathbb{L}_n\left((m_j)_{j=1}^\infty\right)  \leq n^k\right\} \subseteq \bigoplus_{j=1}^{k} \Z/n\Z.
\end{equation}
The right hand side in Equation \eqref{B(n^k)} is understood as the subgroup of $G_n$ consisting of elements 
$(m_j)_{j=1}^\infty \in G_n$ with $m_j=0$ for all $j > k$. 
To prove the inclusion in Equation \eqref{B(n^k)},  take any $(m_j)_{j=1}^\infty \in G_n \setminus \bigoplus_{j=1}^{k-1} \Z/n\Z$, which means precisely that there is $l > k$ such that $m_l \neq 0$. Now using that each 
side in the regular $n^\op{th}$-gon inscribed in the unit circle (that is the one whose vertices are given 
by the $n$-th roots of unity) has length $2\sin(\frac{\pi}{n})$ and that $2\sin(\frac{\pi}{n})>\frac{1}{n}$ (see Lemma \ref{SinLemma}), we find 
\begin{align*}
\mathbb{L}_n\left((m_j)_{j=1}^\infty\right)   & \geq \left| \exp\left(\frac{2\pi im_l}{n}\right)-1 \right|n^l\\
& \geq \left| \exp\left(\frac{2\pi}{n}\right)-1 \right|n^{k+1} \\
& = 2\sin\left(\frac{\pi}{n}\right)n^{k+1} \\
& > \frac{1}{n}n^{k+1}\\
& =n^{k},
\end{align*}
so the inclusion in Equation \eqref{B(n^k)} holds and therefore $\op{card}(B_{\LL_n}(n^k)) \leq n^{k}$, 
which implies that $\LL_n$ is indeed proper. 

Finally, we claim that for any $k \in \Z_{\geq 1}$ and any $l \in \Z_{\geq 1}$
\begin{equation}\label{n-dil}
\op{card}(B_{\LL_n}(n^{k+l})) \leq n^{2+l}\op{card}(B_{\LL_n}(n^k)).
\end{equation}
When $k \in \{1,2\}$, it is clear that $\op{card}(B_{\LL_n}(n^k)) \geq 1$ and therefore 
\[
\frac{\op{card}(B_{\LL_n}(n^{k+l}))}{\op{card}(B_{\LL_n}(n))} \leq n^{k+l} \leq n^{2+l},
\]
so \eqref{n-dil} holds for $k \in \{1,2\}$. Now for any $k \in \Z_{\geq 3}$, 
take any $(m_j)_{j=1}^{\infty} \in \bigoplus_{j=1}^{k-2} \Z/n\Z \subset G$ and notice that 
\begin{align*}
\mathbb{L}_n\left((m_j)_{j=1}^\infty\right) & = 
\sum_{j=1}^{k-2} \left| \exp\left(\frac{2\pi im_j}{n}\right)-1 \right|n^j \\
& \leq 
2\sum_{j=1}^{k-2} n^j \\
&= 2 \cdot \left( \frac{n^{k-1}-n}{n-1}\right) \leq n\cdot n^{k-1}\\
&=n^k.
\end{align*}
That is, $\bigoplus_{j=1}^{k-2} \Z/n\Z \subseteq B_{\LL_n}(n^k)$ and therefore $n^{k-2} \leq \op{card}(B_{\LL_n}(n^k))$
for any $k \in \Z_{\geq 3}$. Thus, for any $l \in \Z_{\geq 1}$ and any $k \in \Z_{\geq 3}$ we have 
\[
\frac{\op{card}(B_{\LL_n}(n^{k+l}))}{\op{card}(B_{\LL_n}(n^k))} \leq \frac{n^{k+l}}{n^{k-2}} = n^{2+l},
\] 
proving that \eqref{n-dil} also holds for $k \in \Z_{\geq 3}$, so our claim is proved. 
We will use this claim to check that $\LL_n$ is of bounded $n$-dilation with $K_{\LL_n}=n^4$. To do so, take any any $R \in \R_{\geq 1}$. If $R=1$, it's clear that 
\[
\op{card}(B_{\LL_n}(n)) \leq n\op{card}(B_{\LL_n}(1)) < n^4\op{card}(B_{\LL_n}(1)).
\]
Otherwise, find $k \in \Z_{\geq 1}$ such that $n^k \leq R < n^{k+1}$. Then, $B_{\LL_n}(n^k) \subseteq B_{\LL_n}(R)$ and $B_{\LL_n}(nR) \subseteq B_{\LL_n}(n^{k+2})$, whence \eqref{n-dil} gives 
\[
\frac{\op{card}(B_{\LL_n}(nR))}{\op{card}(B_{\LL_n}(R))} \leq \frac{\op{card}(B_{\LL_n}(n^{k+2}))}{\op{card}(B_{\LL_n}(n^k))} \leq n^{2+2}=n^4.
\]
This gives at once that $\LL_n$ is indeed of bounded $n$-dilation with $K_{\LL_n}=n^4$ and finishes the proof.
\end{proof}

\begin{corollary}\label{corLn}
The function $\mathbb{L}_n \colon G_n \to \R_{\geq 0}$ given by (\ref{LonZ})
is of bounded doubling with $C_{\LL_n} = n^{4(1+\log_n(2))}$ (see Definition \ref{LengthF}).
\end{corollary}
\begin{proof}
Thar $\LL_n$ is of bounded doubling follows from the fact that $\mathbb{L}_n$ has 
bounded $n$-dilation, shown in Proposition \ref{L_nBD},
and the discussion after Definition 5.1 in \cite{LoWu21}. 
That $C_{\LL_n} = n^{4(1+\log_n(2))}$ follows from  \cite[Equation (5.2)]{LoWu21}. 
\end{proof}

Finally, we point out that the group $G_n \times G_n$ where as before $G_n=\bigoplus_{j=1}^\infty \Z/n\Z$ also has a bounded doubling length function on it,  providing another instance of unifying UHF/Group examples of metric spectral triples. 

Indeed, we know from Corollary \ref{corLn} that $\LL_n$ is a proper length function on $G_n$ of bounded doubling. Therefore, Proposition 3.7 in \cite{FaLaLaPa}, gives that $\LL_n^\Sigma \colon G_n \times G_n 
\to \R_{\geq 0}$, defined by 
\[
\LL_n^\Sigma\big( (x,y) \big)= \LL_n(x)+\LL_n(y),
\]
is a proper length function on $G_n \times G_n$ of bounded doubling with $C_{\LL_n^\Sigma}=n^{16(1+\log_n(2))}$. 

\bibliographystyle{plain}
\bibliography{LpST.bib} 

\vfill

\end{document}